\documentclass[12pt]{amsart} 
\usepackage[left=3cm,top=2.5cm,bottom=2.5cm,right=3cm]{geometry}
\usepackage{times}
\usepackage{amssymb, amsmath, amsthm}
\usepackage{amsmath}
\usepackage{amssymb,amsfonts,stmaryrd,mathrsfs}
\usepackage{latexsym,amsthm,mathtools,bbm}
\usepackage{mathtools}
\usepackage{graphicx,xspace}
\usepackage{epsfig}
\usepackage{enumitem}
\usepackage[usenames,dvipsnames]{xcolor}
\usepackage{tikz}
\usepackage{gensymb}
\usepackage[T1]{fontenc}
\usepackage[utf8]{inputenc}
\usepackage{bm}
\usepackage{subcaption}

\definecolor{green}{RGB}{0,127,0}
\definecolor{red}{RGB}{191,0,0}
\usepackage[colorlinks,cite color=red,link color=green,pagebackref=true]{hyperref}
\usepackage[capitalize]{cleveref}
\usepackage{todonotes}

\newtheorem{thm}{Theorem}[section]
\newtheorem{cor}[thm]{Corollary}
\newtheorem{lem}[thm]{Lemma}
\newtheorem{prop}[thm]{Proposition}
\newtheorem{defi}[thm]{Definition}

\newtheorem{example}[thm]{Example}

\newtheorem{rmk}[thm]{Remark}
\newtheorem{exe}[thm]{Example}

\newcommand{\ZZ}{\mathbb{Z}}
\newcommand{\QQ}{\mathbb Q}

\newcommand{\bM}{\bar{M}}

\newcommand{\vweight}{\textup{vw}}
\newcommand{\Valley}{\mathcal{V}}
\newcommand{\VI}{\mathcal{VI}}
\newcommand{\height}{{\textup{ht}}}
\newcommand{\theight}{\widetilde{{\textup{ht}}}}

\newcommand{\sign}{\textup{sign}}

\newcommand{\normH}[1]{\left\|\tH_{#1}\right\|^2_{*}}

\DeclareMathOperator{\ad}{ad}
\DeclareMathOperator{\Aut}{Aut}
\DeclareMathOperator{\SU}{SU}

\DeclareMathOperator{\End}{End}

\newcommand{\aaa}{\mathfrak{a}}
\newcommand{\bbb}{\mathfrak{b}}

\newcommand{\bfgamma}{\boldsymbol{\gamma}}

\newcommand{\bfbeta}{\boldsymbol{\beta}}

\newcommand{\bQ}{\mathbf{Q}}
\newcommand{\bR}{\mathbf{R}}
\newcommand{\mcNI}{\mathcal{NI}}
\newcommand{\mcP}{\mathcal{P}}

\newcommand{\mcQ}{\mathcal{Q}}
\newcommand{\mcR}{\mathcal{R}}

\newcommand{\tH}{\widetilde{H}^{(q,t)}}

\newcommand{\KK}{\mathbb K}

\newcommand{\mcO}{\mathcal O}
\newcommand{\mcT}{\mathcal T}

\newcommand{\YY}{\mathbb Y}

\newcommand{\A}{\mathcal A}

\newcommand{\Qsym}{\mathcal{Q}^{\textup{sym}}}

\title[]{Path operators and $(q,t)$-tau functions}

\author[H.~Ben Dali]{Houcine Ben Dali}

\address{\parbox{\linewidth}{  Department of Mathematics, Harvard University, Cambridge, MA 02138, U.S.A.}
}
\email{bendali@math.harvard.edu}

\author[V.~Bonzom]{Valentin Bonzom}
\address{LIGM, CNRS UMR 8049, Université Gustave Eiffel, Champs-sur-Marne, France}
\email{valentin.bonzom@univ-eiffel.fr}

\author[M.~Dołęga]{Maciej Dołęga}
\address{
Institute of Mathematics, 
Polish Academy of Sciences, 
ul. Śniadeckich 8, 
00-956 Warszawa, Poland.
}
\email{mdolega@impan.pl}

\thanks{This research was funded in whole or in part by {\it Narodowe Centrum Nauki}, grant 2021/42/E/ST1/00162. For the purpose of Open Access, the author has applied a CC-BY public copyright licence to any Author Accepted Manuscript (AAM) version arising from this submission. V.~B.~was partially supported by the ANR-23-CE48-0018 CartesEtPlus, and the ANR-20-CE48-0018 3DMaps and the ANR-21-CE48-0017 LambdaComb.
}
\begin{document}

\begin{abstract}
We construct a new class of operators that act on symmetric functions with two deformation parameters $q$ and $t$. Our combinatorial construction associates each operator with a specific lattice path, whose steps alternate between moving up and down. We demonstrate that positive linear combinations of these operators are the images of Negut elements via a representation of the shuffle algebra acting on the space of symmetric functions. Additionally, we provide a monomial, elementary, and Schur symmetric function expansion for the symmetric function obtained through repeated applications of the path operators on $1$.

We apply path operators to investigate a $(q,t)$-deformation of the classical hypergeometric tau functions, which generalizes several important series already present in enumerative geometry, gauge theory, and integrability. We prove that this function is uniquely characterized by a family of partial differential equations derived from a positive linear combination of path operators. We also use our operators to offer a new, independent proof of the key result in establishing the extended delta conjecture of Haglund, Remmel, and Wilson.
\end{abstract}

\maketitle

\section{Introduction}

Hurwitz theory is a branch of enumerative geometry that focuses on counting branched coverings of the sphere. Various generating functions in Hurwitz theory can be expressed, using monodromy and the representation theory of symmetric groups, as explicit infinite series involving Schur symmetric functions~\cite{CavalieriMiles2016}. This perspective has proven very successful, revealing many beautiful connections between the somewhat forgotten Hurwitz theory and other modern fields such as Gromov--Witten theory, topological recursion, integrable probability, etc.~\cite{Okounkov2000,OkounkovPandharipande2006,KazarianLando2007,EynardMulaseSafnuk2011,Guay-PaquetHarnad2017}. A recurring technique that has turned out to be highly effective involves manipulating generating functions by constructing various operators that act on Schur symmetric functions in a controllable manner~\cite{Okounkov2000a,OrlovScherbin2000,GouldenJackson2008,EynardMulaseSafnuk2011,BychkovDuninBarkowskiKazarianShadrin2020,BonzomChapuyCharbonnierGarcia-Failde2024}.  

Until very recently, several questions have remained open in the deformed setting, where Schur symmetric functions are replaced by their one-parameter deformation given by Jack symmetric functions. The operators acting in classical Hurwitz theory are usually constructed through nested adjoint actions, and a similar construction might be used in the deformed setting~\cite{Lassalle2008b}. However, this approach does not yield an explicit formula for the operators. In~\cite{ChapuyDolega2022}, an explicit formula for several such operators was found, yielding a positive formula for the representation of the generators of the algebra $\mathbf{SH}^>$ introduced by Schiffmann and Vasserot~\cite{SchiffmannVasserot2013}. This was crucial in discovering an enumerative interpretation of Jack-deformed a.k.a. $b$-deformed Hurwitz theory. This positive formula was later reinterpreted in terms of the so-called path operators, an idea that proved to be very fruitful, leading to the proof of Lassalle's conjecture on the positivity of Jack characters~\cite{BenDaliDolega2023} and providing new insights into $b$-Hurwitz theory~\cite{Ruzza2023,ChidambaramDolegaOsuga2024,ChidambaramDolegaOsuga2024b,BenDali25}. The simplest instance of path operators can be recognized as the trace of a matrix of operators, and such an operator had already appeared in the context of Jack symmetric functions in the work of Nazarov--Sklyanin~\cite{NazarovSklyanin2013}. It was later incorporated into the framework of path operators to prove global asymptotics of a large class of discrete $\beta$-ensembles~\cite{Moll2023,CuencaDolegaMoll2023}.  

Nazarov and Sklyanin~\cite{NazarovSklyanin2015} also constructed analogous operators in another deformed case, where Schur symmetric functions are replaced by Macdonald symmetric functions. This deformation consists of two parameters, $q$ and $t$, and its classical limit corresponds to the aforementioned Jack deformation. Motivated by their work and the various applications of path operators in the Jack-deformed case, we extend these ideas to the $(q,t)$-deformation. Here is a summary of our main results.  

\begin{enumerate}  
\item We construct a new class of operators acting on symmetric functions with two deformation parameters $q$ and $t$. Our combinatorial construction associates an operator with a certain lattice path, whose steps alternate between going up and down. For any sequence of integers $\beta \in \mathbb{Z}^\ell$, we define an operator $\mathcal{R}_\beta$ using an explicit formula given by a sum of path operators lying over a special lattice path associated with $\beta$. We demonstrate that this operator is the image of the element $D_\beta$ from the shuffle algebra, constructed by Negut~\cite{Negut2014}, through the representation of the shuffle algebra on symmetric functions.  
\item We find a monomial, elementary, and Schur expansion of the symmetric function $\mathcal{R}_{\beta^{(n)}}\dotsm\mathcal{R}_{\beta^{(1)}}\cdot 1$, constructed by the repeated action of path operators on $1$.  
\item We define certain formal series that we call \emph{the $(q,t)$-tau function}, a natural $(q,t)$-deformation of the tau functions for the $G$-weighted $b$-Hurwitz numbers of \cite{ChapuyDolega2022}. Various special cases of this function have appeared as important generating series in enumerative geometry and mathematical physics, such as in the context of mixed Hodge polynomials of character varieties of the Riemann sphere~\cite{HauselLetellierRodriguezVillegas2011}, the celebrated AGT conjecture~\cite{Yanagida2016}, and in the recent work of Bourgine and Garbali~\cite{BourgineGarbali2024}, as a tau function of a $(q,t)$-extension of the 2D Toda hierarchy. We prove that the $(q,t)$-tau function satisfies a system of partial differential equations (PDEs) of a combinatorial nature built from our path operators. Moreover, we demonstrate these equations uniquely determine the $(q,t)$-tau function.  
\item We apply our techniques to provide a new, independent proof of the key result \sloppy \cite[Theorem 4.4.1]{BlasiakHaimanMorsePunSeelinger2023b} in proving the extended delta conjecture of Haglund, Remmel and Wilson~\cite{HaglundRemmelWilson2018}. We believe that this approach may be used in the future to tackle related open problems in the field.  
\end{enumerate}

\subsection{Path operators}

Let $N=(x_N,y_N)$ be a point of $\ZZ_{\geq 0}\times \ZZ$. We call \emph{a path} $\gamma$ from the origin $(0,0)$ to $N$, a sequence of points $(w_{0}, w_{1},\dotsc,w_{x_N})$ in $\ZZ_{\geq 0}\times \ZZ$ such that $w_j=(j,y_j)$, with $w_{0}=(0,0)$ and $w_{x_N}=N$. Such a path is uniquely encoded by its sequence of \emph{steps} $\gamma_j \coloneqq y_j-y_{j-1} \in\ZZ$. The \emph{length} of $\gamma$ is $x_N$ and its \emph{degree} is $y_N$.

We say that a path $\gamma=(\gamma_1,\dotsc, \gamma_{2n})$ from $(0,0)$ to $N$ is \emph{alternating},
if $\gamma_{2j-1}\geq 0$ and $\gamma_{2j}\leq 0$ for any $0\leq j\leq n$. 
In other terms, odd steps are up steps and even steps are down steps,
with the convention that a flat step is considered either an up or a down
step depending on its parity.

We say that a point $V=(x_V,y_V)$ of $\gamma$ is a \emph{valley} of $\gamma$ if $x_V$ is even. This means that $V$ is preceded by a down step and followed by an up step, or $V$ is the start or the end point of the path. Similarly, we say that $P=(x_P,y_P)$ is a \emph{peak} if $x_P$ is odd.

For $\beta = (\beta_1, \dotsc, \beta_\ell)\in\mathbb{Z}^\ell$, let $\gamma_\beta$ be the alternating path of length $2\ell$ such that for $i=1, \dotsc, \ell$, its steps are
\[\gamma_{2i-1} = \max(0,\beta_i),\qquad
\gamma_{2i} = \min(0, \beta_i).
\]
Equivalently, for $\beta_i\geq0$ (resp. $\beta_i\leq 0$) the $i$-th up step is $\beta_i$ (resp. is horizontal) and the $i$-th down step is horizontal (resp. is $\beta_i$).

We denote $\bR_\beta$ the set of alternating paths that start and end at the same points as $\gamma_\beta$ (i.e. have the same length and degree) and that stay weakly above $\gamma_\beta$. Notice that the valleys of $\gamma_\beta$ have $y$-coordinates $y_{2j} = \sum_{i=1}^j \beta_i$ for $j=0, \dotsc, \ell$. 

For $\gamma\in\bR_\beta$ and  $j=0, \dotsc, \ell$, let the \emph{$\beta$-height} of a valley $V=(2j, y_{2j})$ be
\begin{equation}\label{eq:def_height}
  \height_\beta(V) = y_{2j} - \sum_{i=1}^{j}\beta_i.  
\end{equation}
By definition of $\bR_\beta$, one has $\height_\beta(V)\geq0$ for all valleys $V$ of $\gamma$. Moreover, the start and the end point of $\gamma$ always have $\beta$-height zero.

\begin{figure}[h]
	\centering
	\includegraphics[width = 0.5\linewidth]{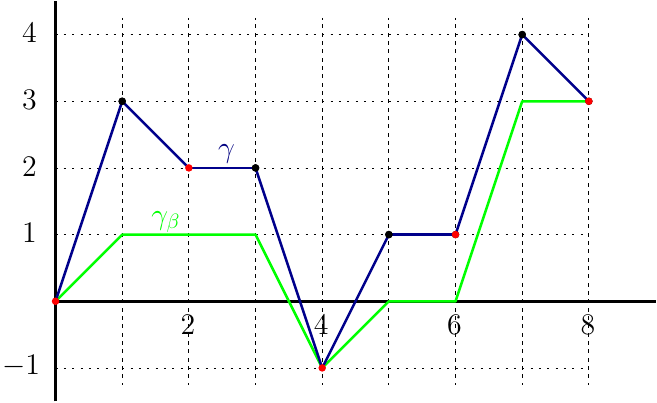}
	\caption{Two alternating paths of length 8 and degree 3: $\gamma$ (in blue) and $\gamma_\beta$ (in green) with $\beta = (1,-2,1,3)$. The five valleys of $\gamma$ are represented in red.}
	\label{fig:alternating_path}
\end{figure}

\begin{example}
	We give in \cref{fig:alternating_path} an example of an alternating path
	$\gamma$ of length $8$ and degree 3. We also draw the alternating path $\gamma_\beta$ for $\beta = (1,-2,1,3)$ to show that $\gamma \in \bR_{(1,-2,1,3)}$.
\end{example}

We denote by $\Lambda$ the vector space of symmetric functions over the fraction field $\QQ(q,t)$ in an alphabet of variables $X:=x_1+x_2+\dotsm$. 
\begin{defi}\label{def:decorated_paths}
	Let $\gamma\in\bR_\beta$ be an alternating
	path. We associate to $\gamma$ the \emph{valley weight}
	\begin{equation}\label{eq:valley_weight_PMAP}
		\vweight_\beta(\gamma)=\prod_{V\in\Valley(\gamma)}(qt)^{\height_\beta(V)},  
	\end{equation}
	the product being taken over the valleys of $\gamma$. We associate to $\gamma$ the linear operator $\mcO(\gamma)\colon \Lambda \rightarrow \Lambda$ as follows:  if $\gamma$ has steps $(\gamma_1,\dots,\gamma_{2l})$ then 
	\begin{equation*}
		\label{eq:decorated_path}
		\mcO_\beta(\gamma):=\vweight_\beta(\gamma)\mcO(\gamma_1)\dots \mcO(\gamma_{2l}),
	\end{equation*}
	where for any integer $m$ we define
	\begin{equation}\label{eq:def_step_operator}
		\mcO(m):=\left\{
		\begin{array}{ll}
			(-1)^me_{m}[X]=h_{m}[-X]   & \text{ if } m>0  \\
			h^\perp_{-m}[MX]   & \text{ if } m<0\\
			1 &\text{ if } m=0,
		\end{array}
		\right.  
	\end{equation}
    	with $M:= (1-q)(1-t)$.
	Finally, we define  
	\begin{equation*}\label{eq:def_path_operators}
		\mcR_{\beta}:=\sum_{\gamma\in\bR_{\beta}}\mcO_\beta\left(\gamma\right).
	\end{equation*}	
	
\end{defi}
We use a standard notation where square brackets $[\cdot]$ denote the plethystic substitution, the operator $e_{m}$ is to be understood as the multiplication by the symmetric function $e_{m}$, and similarly $h_m^\perp$ denotes the adjoint of the multiplication by $h_m$ with respect to the Hall scalar product; see~\cref{sec:Prelim} for more details.

\begin{rmk}
	\label{rem:remark_path_op}
	Since the set $\bR_{\beta}$ is infinite, it might not be immediately clear that $\mcR_{\beta}\colon \Lambda \rightarrow \Lambda$ is a well-defined endomorphism of $\Lambda$. It is indeed well-defined since for any $f \in \Lambda$ there are only finitely many paths $\gamma \in \bR_{\beta}$ such that $f \notin \ker(\mcO_\beta(\gamma))$. This is a consequence of the property that for any sequence $(\gamma_1,\dots,\gamma_\ell) \in \ZZ^\ell$ such that the partial sum $\sum_{i =j}^\ell \gamma_i$ is smaller than $-m$ for some $1 \leq j \leq \ell$, the operator $\mcO(\gamma_1)\cdots \mcO(\gamma_\ell)$ annihilates the space $\Lambda^i$ of homogeneous symmetric functions of degree $i \leq m$:
	\[  \bigoplus_{i=0}^m \Lambda^i \subset \ker \left( \mcO(\gamma_1)\cdots \mcO(\gamma_\ell)\right).\] 
\end{rmk}

\begin{example}
	The operator associated to the path $\gamma$ of \cref{fig:alternating_path} is given by
	\[\mcO_{(1,-2,1,3)}(\gamma)=(qt)^{1+0+1} \cdot(-1)^{3+0+2+3}e_{3}[X]h^\perp_{1}[MX]h^\perp_{3}[MX]e_{2}[X]e_{3}[X]h^\perp_{1}[MX].\]
\end{example}

\subsection{Product formula in terms of \texorpdfstring{$D$}{D}-operators and connection to the shuffle algebra}

We prove that the path operators $\mcR_\beta$ have a simple expression using the vertex operator algebra (VOA) formalism involving the series $D(z)$ introduced in~\cite{GarsiaHaimanTesler1999}. As a corollary we deduce that the path operators are the images of the so-called Negut elements via a representation of the shuffle algebra acting on the space of symmetric functions. Consequently, we find a new explicit combinatorial formula for this representation of the Negut elements.

We say that a formal expression
\[ F(z) := \sum_{n \in \ZZ}a_n z^{n} \in \End[\Lambda]\llbracket z,z^{-1}\rrbracket\]
is a \emph{field} if for any $f \in \Lambda$, there exists $N$ such that $f \in \ker(a_n)$ for all $n < N$ (see \cite{Kac2017}). It is well-known (one can use a similar reasoning as in \cref{rem:remark_path_op}) that whenever $F_1,\dots,F_\ell$ are fields, then the coefficients of the following Laurent series in $z_1,\dots,z_\ell$ are well-defined endomorphisms of $\Lambda$:
\[ \frac{F_1(z_1) \dotsm F_\ell(z_{\ell})}{\prod_{i=1}^{\ell-1} \bigl(1-z_{i+1}/z_i\bigr)} \in \End[\Lambda]\llbracket z_1,z_1^{-1},\dots,z_\ell,z_\ell^{-1}\rrbracket, \quad \bigl(1-z_{i+1}/z_i\bigr)^{-1} := \sum_{n \geq 0}z_{i+1}^n \cdot z_i^{-n}.\]
We recall (see~\cref{sec:Prelim} for more details) that the field $D(z)$ is defined as follows:
\[ D(z) := \sum_{n \in \ZZ} D_n z^n = \sum_{n\geq 0} (-1)^n e_n[X] z^n \cdot \sum_{n\geq 0} h_n^\perp[(1-q)(1-t)X] z^{-n} = \sum_{n \in \ZZ} \mcR_{(n)} z^n,\]
and it satisfies the following commutation relation:
\begin{equation} \label{CommutationD}
	\omega(z_1/z_2) D(z_1) D(z_2) = \omega(z_2/z_1) D(z_2) D(z_1),
	\quad \omega(x) := \frac{(1-x)(1-qtx)}{(1-qx)(1-tx)}.
\end{equation}
Following this formalism, we prove the following VOA expression for the path operators.
\begin{thm}\label{thm:R_vertex_op}
	Fix $\beta=(\beta_1,\dots,\beta_\ell)\in\ZZ^{\ell}$, then
	\begin{equation}\label{eq:R_vertex_formula}
		\mcR_\beta = [z_1^{\beta_1} \dotsm z_{\ell}^{\beta_{\ell}}] \frac{D(z_1) \dotsm D(z_{\ell})}{\prod_{i=1}^{\ell-1} \bigl(1-qt z_{i+1}/z_i\bigr)}.
	\end{equation}
\end{thm}

The shuffle algebra $S$ \cite{FeiginTsymbaliuk2011,Negut2014} is the graded algebra $S = \oplus_{n\geq 0} S_n$ where $S_0 = \QQ(q,t)$,  and for $n\geq 1$ an element $F$ of $S_n$ is of the form
\begin{equation*}
	F(z_1, \dotsc, z_n) = f(z_1, \dotsc, z_n) \prod_{1\leq i\neq j\leq n} \frac{z_i-z_j}{(z_i-qz_j)(z_i-tz_j)},
\end{equation*}
where $f(z_1, \dotsc, z_n) \in \QQ(q,t)[z_1, z_1^{-1}, \dotsc, z_n,z_n^{-1}]^{\mathfrak{S}_n}$ is a symmetric Laurent polynomial which satisfies the \emph{wheel condition}
\begin{equation*}
	f(z_1, \dotsc, z_n) = 0\qquad \text{whenever} \quad \left\{\frac{z_1}{z_2}, \frac{z_2}{z_3}, \frac{z_3}{z_1}\right\} = \left\{q, t, \frac{1}{qt}\right\}.
\end{equation*}
In particular, $S_1 = \QQ(q,t)[z_1,z_1^{-1}]$.

The shuffle algebra $S$ is equipped with the following product: for $F\in S_n, G\in S_m$, then $F*G\in S_{n+m}$ with
\begin{equation*}
	(F*G)(z_1, \dotsc, z_{n+m}) = \frac{1}{n!m!} \operatorname{Sym}\left( F(z_1, \dotsc, z_n) G(z_{n+1}, \dotsc, z_{n+m}) \prod_{i=1}^n \prod_{j=n+1}^{n+m} \omega(z_j/z_i)\right),
\end{equation*}
where $\operatorname{Sym}$ denotes the symmetrization with respect to the variables $z_1, \dotsc, z_{n+m}$.

We will use the following action of the shuffle algebra on the ring of symmetric functions. We refer to \cite{FeiginHashizumeHoshinoShiraishiYanagida2009,FeiginTsymbaliuk2011} for more details.
\begin{defi}\label{def:shuffle_algebra}
The shuffle algebra $S$ acts on $\Lambda$ through the following representation $F\mapsto \hat{F}$ 
which maps $F\in S_n$ to an operator $\hat{F} \in \End(\Lambda)$ and defined on the generators $z^a$ by
\begin{equation*}
    \widehat{z^a}=D_{-a}=[z^0]z^aD(z),\quad \forall a\in \ZZ.
\end{equation*}
It follows then from \cref{CommutationD} that for any $F\in S$,
\begin{equation*}
	\hat{F}:= [z_1^{0} \dotsm z_n^{0}] F(z_1, \dotsc, z_n) \frac{D(z_1) \dotsm D(z_n)}{\prod_{1\leq i<j\leq n} \omega(z_j/z_i)}.
\end{equation*}
\end{defi}

 In \cite{Negut2014}, Negut introduced specific elements
\begin{equation*}
	X_{\beta}(z_1, \dotsc, z_\ell):= \operatorname{Sym}\left(\frac{z_1^{-\beta_1} \dotsm z_\ell^{-\beta_\ell}}{\prod_{i=1}^{\ell-1} (1-qtz_{i+1}/z_i)} \prod_{1\leq i<j\leq \ell} \omega(z_j/z_i)\right) \quad \in S_\ell,
\end{equation*}
for $\beta = (\beta_1, \dotsc, \beta_\ell)\in \ZZ^\ell$. Using the relations of \cref{CommutationD} of the $D$-operators, a straightforward calculation shows that $\hat{X}_\beta$ coincides with the RHS of \cref{eq:R_vertex_formula}. Thus, we have the following corollary from \Cref{thm:R_vertex_op}.
\begin{cor}
\label{cor:Negut}
Fix $\beta = (\beta_1, \dotsc, \beta_\ell)\in\ZZ^\ell$. Then we have
\[	\hat{X}_\beta = \mcR_\beta. \]
\end{cor}

Blasiak {\it et al} \cite{BlasiakHaimanMorsePunSeelinger2023,BlasiakHaimanMorsePunSeelinger2023b} used these distinguished elements (also called Negut elements) to prove some generalizations of the shuffle and delta conjectures. Remarkably, our analysis of path operators leads to a combinatorial proof of Theorem 4.4.1 from \cite{BlasiakHaimanMorsePunSeelinger2023b}, which is the authors' key reformulation of the symmetric function side of the extended delta conjecture. It comes out as a corollary of the following theorem, which itself follows from \Cref{thm:R_vertex_op}.

\begin{thm}\label{thm:explicit_formula}
	Fix $\beta=(\beta_1,\dots,\beta_\ell)$ of size $n\geq 0$. 
	For any partition $\lambda$ of size $n$ and for the alphabet $Z =
	z_1+\dotsm+z_\ell$ we have
	\begin{multline}
		(-1)^n [s_\lambda[X]] \mcR_{\beta}\cdot 1\\
		=\left[z_1^{\beta_1} \dotsm z_\ell^{\beta_\ell}\right]
		s_{\lambda'}[Z]\prod_{1\leq i<j\leq \ell} \frac{1-z_j/z_i}{(1-qz_j/z_i)(1-tz_j/z_i)}
		\prod_{1\leq i<j-1<\ell} (1-qtz_j/z_i),
	\end{multline}
	where 
	$\lambda'$ is the transpose of $\lambda$, and 
	the rational functions in the RHS should be expanded as formal series in $z_j/z_i$ for $j>i$.
\end{thm}	
Similar formulas hold for $(-1)^n [m_\lambda[X]] \mcR_{\beta}\cdot 1$ and $(-1)^n[e_\lambda[X]] \mcR_{\beta}\cdot 1$, as given in \cref{subsec:NormalOrdering}.

We will also give versions of \Cref{thm:R_vertex_op} and \Cref{thm:explicit_formula} for products of path operators, i.e $\mcR_{\bfbeta} \coloneqq \mcR_{\beta^{(1)}}\dotsm \mcR_{\beta^{(\ell)}}$, with $\bfbeta = (\beta^{(1)}, \dotsc, \beta^{(\ell)})$ where each $\beta^{(i)}\in\ZZ^{\ell_i}$ is a sequence of integers.

\subsection{The \texorpdfstring{$(q,t)$}{(q,t)}-tau function}
Consider two infinite sequences of variables $U = u_1,u_2,\dotsc$ and $V = v_1,v_2,\dotsc$, and define $\KK$ as the ring of formal power series in these variables with coefficients in $\QQ(q,t)$ 
\[	\KK:=\QQ(q,t)\llbracket U,V \rrbracket. \]
Let  
\[G_1(\hbar) := 1+\sum_{n=1}^\infty u_n \hbar^n \in \QQ[U,V]\llbracket \hbar \rrbracket,\quad  G_2(\hbar) :=  1+\sum_{n=1}^\infty v_n \hbar^n \in \QQ[U,V]\llbracket \hbar \rrbracket.\]
Let $G(\hbar) := G_1(\hbar) \cdot G_2^{-1}(\hbar) \in \QQ[U,V]\llbracket \hbar \rrbracket$, and note that the substitution $ G(q^{j} t^{i}) := \frac{G_1(q^{j} t^{i})}{G_2(q^{j} t^{i})} \in \KK$ is well-defined for all $i,j\geq 0$. The \emph{$G$-weighted $(q,t)$-tau function} is the series in $\Lambda^\KK_X\otimes\Lambda^\KK_Y\llbracket z\rrbracket$ defined by:
\begin{equation*}
	\label{eq:GWeightedDoubleTau}
	\tau_G(z,X,Y) :
	= \sum_{\lambda\in\mathbb
		Y}z^{|\lambda|}\frac{\tH_\lambda[X]\tH_\lambda[Y]}{\normH{\lambda}}\prod_{(i,j)\in\lambda} G(q^{j-1} t^{i-1}),
\end{equation*}
where $\tH_\lambda$ are modified Macdonald polynomials, and $\Lambda^\KK_A$ is the algebra of symmetric functions in the alphabet $A = X,Y$ with coefficients in $\KK$. The notation $\normH{\lambda}$ is the squared norm of the modified Macdonald polynomial $\tH_\lambda$  (see \cref{sub:MacdonaldPrel} for the details).

Beyond it being a natural $(q,t)$-deformation of the tau
functions for the $G$-weighted $b$-Hurwitz numbers of \cite{ChapuyDolega2022}, this function has appeared as an important generating series in enumerative geometry and
mathematical physics for specific cases of the weight $G$. When $G(\hbar) = \prod_{i=1}^n(1-\hbar u_i)$, the plethystic
logarithm of $\tau_G(z,X,Y)$ was conjectured by
Hausel--Lettelier--Rodriguez-Villegas~\cite{HauselLetellierRodriguezVillegas2011} to be the generating series of
the mixed Hodge polynomials of character varieties of the Riemann
sphere. When $G(\hbar) = (1-u\hbar)^{-1}$, the specialization $Y =
1$ corresponds to the Whittaker vector for the
deformed Virasoro algebra from the $5D$ $\mathcal{N} = 1$ pure
$\SU(2)$ case of the AGT conjecture~\cite{Yanagida2016}. Finally, the
case $G(\hbar)=\hbar$ has appeared in the recent work of
Bourgine and Garbali~\cite{BourgineGarbali2024} as a tau function of a $(q,t)$-extension of the
2D Toda hierarchy\footnote{In
	\cite{BourgineGarbali2024,Yanagida2016} the authors worked with the
	$P$-version of Macdonald polynomials. We define our $(q,t)$-tau
	function in terms of the modified Macdonald polynomials, as such a function seems
	to have much more interesting combinatorial structure that is
	partially supported by the conjecture of Hausel--Lettelier--Rodriguez-Villegas.}. 
	
Our main result is a construction of a system of functional equations which uniquely characterizes the $G$-weighted $(q,t)$-tau function. These equations involve some families of path operators parametrized by non-negative sequences $\beta$. It will be convenient to describe these operators by the distances between unit increments, rather than the number of unit increments in each position given by $\beta$. This list of distances, which we denote by $\psi(\beta)$, is obtained from $\beta$ by applying a simple reparametrization bijection $\psi$ (see \cref{sub:path_op} for a precise definition). We then denote by $\mcQ$ the operators obtained from $\mcR$ by this reparametrization:
\[\mcQ_{\psi(\beta)} := \mcR_\beta.\]
Then for $F(\hbar) =1+\sum_{i\geq 1} a_i \hbar^i \in \{G_1, G_2\}$, let
\begin{equation}\label{eq:PathOperatorsA}
  \A_F^{(\ell)} = \sum_{\alpha} a_\alpha \mcQ_\alpha, 
\end{equation}
where the sum is taken over $\alpha \in \ZZ_{\geq 0}^{\ell}$, and $a_\alpha := \prod_{i=1}^\ell a_{\alpha_i}$. We have the following characterization of the $G$-weighted $(q,t)$-tau function.

\begin{thm}\label{thm:dif_eq'}
	For any $\ell\geq 1$ we have
	\begin{equation}\label{eq:diff_eq}
		z^\ell\A^{(\ell)}_{G_1}(X)\cdot\tau_G(z,X,Y)=\left(\A^{(\ell)}_{G_2}(Y)\right)^*\cdot\tau_G(z,X,Y),
	\end{equation}
	where $\left(\A_{G_2}^{(\ell)}\right)^*$ is the adjoint of $\A_{G_2}^{(\ell)}$ with respect to the star scalar product (see \cref{sub:MacdonaldPrel}).
	Moreover, $\tau_G(z,X,Y)$ is the unique function $\Phi_G\in\Lambda^\KK_X\otimes\Lambda^\KK_Y\llbracket z\rrbracket$ that satisfies these equations and such that $[z^0]\Phi_G(z,X,Y)=1$.
\end{thm}
We briefly describe the main steps of the proof of this theorem:
\begin{itemize}
    \item We prove that all the operators $\A_F^{(\ell)}$ can be recursively built from commutators of two elementary operators, $e_1$ and $D_0$. This part of the proof is purely combinatorial and will be carried out in \cref{sec:path_op}.
    \item We use the commutation relations established in the first step and the fact that the operators $e_1$ and $D_0$ act nicely on the function $\tau_G$ to prove \cref{eq:diff_eq} by induction on $\ell$. This will be the main goal of \cref{sec:fun_eq}.
\end{itemize}

\subsection{Outline of the paper}
In Section \ref{sec:PathOperators}, we describe path operators in details and prove some explicit expressions such as those of \Cref{thm:R_vertex_op} and \Cref{thm:explicit_formula}. In Section \cref{sec:fun_eq}, we introduce the $G$-weighted $(q,t)$-tau function and prove a version of the functional equation \eqref{eq:diff_eq} where the operators acting on the LHS and RHS are not defined in terms of path operators. In Section \ref{sec:path_op} we conclude the proof of \Cref{thm:dif_eq'} by showing that these operators are indeed the path operators defined in \cref{eq:PathOperatorsA}.

\section{Path operators and the \texorpdfstring{$D$}{D}-operator} \label{sec:PathOperators}
\subsection{Preliminaries}
\label{sec:Prelim}

Let $\Lambda$ denote the algebra of symmetric functions in the alphabet $X = x_1+x_2+\dotsb$ with coefficients in $\QQ(q,t)$. We denote the classical bases such as monomial, elementary, homogeneous, power-sum, and Schur symmetric functions by $m_\lambda,e_\lambda,h_\lambda,p_\lambda, s_\lambda$, where $\lambda$ is an integer partition, which we denote by $\lambda \in \YY$. The \emph{dominance order} $\leq$ is the partial order on $\YY$ defined by 
$$\mu\leq\lambda \iff |\mu|=|\lambda| \text{ and }\hspace{0.3cm} \mu_1+...+\mu_i\leq \lambda_1+...+\lambda_i \text{ for } i\geq1.$$

For $i\geq 1$, we denote $m_i(\lambda)$ the number of parts of size $i$ in $\lambda$. We then set 
\begin{equation*}\label{eq zlambda}
z_\lambda:=\prod_{i\geq1}m_i(\lambda)!i^{m_i(\lambda)}.  
\end{equation*}

The \emph{Hall scalar product}  $\langle.,.\rangle$ on $\Lambda$ is defined by 
\begin{equation*}\label{eq scalar product}
  \langle p_\mu,p_\nu\rangle=\delta_{\mu,\nu}z_\mu, \qquad \text{for any $\mu,\nu\in\YY$,} 
\end{equation*}
where $\delta_{\mu,\nu}$ is the Kronecker delta. We also have $\langle m_\mu,h_\nu\rangle=\delta_{\mu,\nu}$, for  $\mu,\nu\in\YY$. If $f\in\Lambda$, then we denote $f^\perp$ the dual of the multiplication by $f$, with respect to this scalar product.

For any symmetric function $f \in \Lambda$ and an expression $A$ involving indeterminates, we define the plethystic substitution $f[A]$ by first defining it on the basis of power-sum symmetric functions $p_k[A]$ by the substitution $a \mapsto a^k$ for each indeterminate $a$ occurring in $A$, and then extending on the whole algebra of symmetric functions so that $f \mapsto f[A]$ is a homomorphism. For instance $p_k[X] := \sum_{i \geq 1} x_i^k$ and 
\begin{align*}
	p_k[aX] &= a^kp_k[X] \quad \text{for $a\in\{q,t\}$}\\
	p_k[cX] &= cp_k[X] \quad \text{for $c\in\QQ$}.
\end{align*} 
An expression we will encounter often is $M:=(1-q)(1-t) \in \QQ(q,t)$ and hence $p_k[MX] = (1-q^k)(1-t^k)p_k[X]$.

For an expression A, define the associated \emph{plethystic exponential} by
\[\Omega[A]: = \sum_{n\geq 0} h_n[A] = e^{\sum_{i>0}\frac{p_i[A]}{i}}.\]
In particular, the Cauchy and dual Cauchy kernels are respectively given by
$$\Omega[X]=\sum_{n\geq 0} h_n[X] = \prod_{k\geq 1}\frac{1}{1-x_k},\qquad\text{and}\qquad\Omega[-X]=\sum_{n\geq 0} (-1)^n e_n[X] = \prod_{k\geq 1}(1-x_k).$$

Following \cite{GarsiaHaimanTesler1999}, we consider the operators that act on $f[X]\in\Lambda$ like
\begin{equation*}
	\begin{aligned}
	\mcP_Y f[X] &= \Omega[XY] f[X],\\
	\mcT_Y f[X] &=\sum_{k\geq 0}h^\perp_{k}[Y]f[X]= f[X+Y].
	\end{aligned}
\end{equation*} 
We have the following commutation relations between operators acting on
$\Lambda_X$ (see \cite[Theorem 1.1 and Proposition 1.2]{GarsiaHaimanTesler1999}):
\begin{equation}\label{eq:mcT-mcP'}
 \mcT_Y\cdot \mcP_Z=\Omega[YZ]\mcP_Z\cdot \mcT_Y,   
\end{equation}
An important role in the theory of Macdonald polynomials is played by the operators $D_k$, $k\in\mathbb{Z}$ from \cite{GarsiaHaimanTesler1999}. They are defined by the following equality of formal series
\[D(z) = \sum_{k\in\mathbb{Z}} D_k z^k = \mcP_{-z} \mcT_{M/z}.\]
Since $\mcP_{-z} = \sum_{n\geq 0} (-1)^n e_n[X] z^n$ and $\mcT_{M/z} = e^{\sum_{i>0} \frac{z^{-i}}{i} p_i^\perp[MX]} = \sum_{n\geq 0} h_n^\perp[MX] z^{-n}$, it comes that
\begin{equation}
\label{eq:Dk}
	D_k = \sum_{\substack{m, n\geq 0\\ m-n=k}} (-1)^m e_m[X] h_n^\perp[MX].
\end{equation}
In particular, $D_k$ is equal to the path operator $\mcR_{(k)}$ defined in \cref{def:decorated_paths}:
\begin{equation} \label{2steps}
	\mcR_{(k)} = \sum_{\substack{m, n\geq 0\\ m-n=k}} \mcO(m) \mcO(-n) = \sum_{\substack{m, n\geq 0\\ m-n=k}} (-1)^m e_m[X] h_n^\perp[MX] = D_k.
\end{equation}

\subsection{Decomposition at fixed valley heights and proof of Theorem \texorpdfstring{\ref{thm:R_vertex_op}}{}}
Fix a sequence $\beta=(\beta_1,\dots,\beta_\ell)\in\mathbb{Z}^\ell$. If $\gamma\in\bR_\beta$ has steps $(\gamma_1, \dotsc, \gamma_{2\ell})$, define
\[r_j(\gamma) = \gamma_{2j-1}+\gamma_{2j},
\]
for $j=1, \dotsc, \ell$. These quantities are the increments between the heights of consecutive valleys (including the endpoint). Equivalently, if $\gamma$ has points $(w_0, \dotsc, w_{2\ell})$ with coordinates $w_j = (j, y_j)$, then $r_j(\gamma) = y_{2j}-y_{2j-2}$ for $j=1, \dotsc, \ell$. Clearly, the set of possible values $(r_1,\dots,r_\ell)\in\ZZ^\ell$ such that there exists a path $\gamma\in\mcR_\beta$ with $r_j(\gamma) = r_j$ for all $j=1, \dotsc, \ell$ is
\begin{multline}\label{eq:VI}
	\VI_\beta=\left\{(r_1,\dots,r_\ell)\in\ZZ^\ell\big| \sum_{1\leq i\leq j}r_i\geq \sum_{1\leq i\leq j}\beta_i \text{ for $1\leq j< \ell$,} \text{ and }\sum_{1\leq i\leq \ell}r_i= \sum_{1\leq i\leq \ell}\beta_i.\right\}  
\end{multline}
and the set of corresponding paths is denoted $\bR_\beta(r_1, \dotsc, r_\ell)$. The valley weights of those paths, as defined in \cref{eq:valley_weight_PMAP}, only depend on $(r_1,\dots,r_\ell)$, so we define
\[\vweight_\beta(r_1,\dots,r_\ell) \coloneqq \vweight_\beta(\gamma) = \prod_{1\leq j\leq \ell}(qt)^{\sum_{1\leq i\leq j}(r_i-\beta_i)}.
\]
for any path $\gamma\in\bR_\beta(r_1, \dotsc, r_\ell)$.
\begin{exe}
    The list of increments of the valley heights of the path $\gamma$ given in \cref{fig:alternating_path} is $(2,-3,2,2)$. Since $\gamma\in \bR_{(1,-2,1,3)}$, this implies that $(2,-3,2,2)\in\VI_{(1,-2,1,3)}$.
\end{exe}
We have the following formula for the operator $\mcR_\beta$, where instead of summing over all paths we sum over valley increments.
\begin{prop}\label{prop:Qalpha_D}
	For any $\beta\in \ZZ_{\ell}$,
	\[\mcR_\beta = \sum_{(r_1,\dotsc,r_\ell)\in\VI(\beta)} \vweight_{\beta}(r_1,\dots,r_\ell) [z_1^{r_1}\dotsm z_\ell^{r_\ell}]D(z_1)\dotsm D(z_\ell).
	\]
\end{prop}

\begin{proof}
	We start from
	\[
	\begin{aligned}
		\mcR_\beta &= \sum_{(r_1,\dotsc,r_\ell)\in\VI(\beta)} \sum_{\gamma\in\bR_\beta(r_1, \dotsc, r_\ell)} \vweight_\beta(\gamma) \mcO(\gamma_1)\dotsm \mcO(\gamma_{2\ell})\\
		&= \sum_{(r_1,\dotsc,r_\ell)\in\VI(\beta)} \vweight_\beta(r_1, \dotsc, r_\ell) \sum_{\gamma\in\bR_\beta(r_1, \dotsc, r_\ell)} \mcO(\gamma_1)\dotsm \mcO(\gamma_{2\ell}).
	\end{aligned}
	\]
	Then we can perform the sum over the heights of peaks at fixed valley heights by using the following (see \cref{2steps})
	\[\sum_{\substack{\gamma_{2j-1}\geq 0, \gamma_{2j}\leq 0\\ \gamma_{2j-1} + \gamma_{2j}= r_j}} \mcO(\gamma_{2j-1}) \mcO(\gamma_{2j}) = D_{r_j} = [z^{r_j}] D(z).\qedhere
	\]
\end{proof}

We will use the following straightforward lemma.
\begin{lem}\label{lem:LargerPart}
	Let $n\in\ZZ$, and let $\mcO(z)=\sum_{i\in\ZZ}z^i\mcO_i$ be a field. Then
	\[ [z^n] \frac{\mcO(z)}{1-qt\frac{w}{z}} = \sum_{i\geq n}(qtw)^{i-n} \mcO_i \]
	where the denominator on the LHS is a series in $w/z$.
\end{lem}

\begin{proof}[Proof of \cref{thm:R_vertex_op}]
	We will prove the theorem by induction on $\ell$. The case $\ell=1$ corresponds to \cref{2steps}, so suppose that $\ell >1$ and define for $1\leq m< \ell$ the set of all possible $m$ first valley increments
	\[\VI^{(m)}_\beta = \left\{(r_1, \dotsc, r_m)\in\ZZ^m \big| \sum_{1\leq i\leq j} r_i \geq \sum_{1\leq i\leq j} \beta_k\text{ for } 1\leq j\leq m\right\}. \]
	Let us consider for $1\leq m\leq \ell-1$ the intermediate operators
	\[\mcR_{\beta}^{(m)}(w) = \sum_{(r_1, \dotsc, r_m)\in\VI_\beta^{(m)}} D_{r_1} (qt)^{r_1-\beta_1} D_{r_2} (qt)^{r_1+r_2-\beta_1-\beta_2} \dotsm D_{r_{m}} (qtw)^{\sum_{k=1}^{m} (r_k - \beta_k)}.\]
	Such an operator can be thought of as the operator obtained by considering only the possible first $2m$ steps of paths in $\bR_\beta$ counted with an extra weight $w^j$, where $j\geq 0$ is the $\beta$-height of the $m$-th valley.
	
	For $m=1$, we have $$\mcR_{\beta}^{(1)}(z_2) = \sum_{r_1\geq \beta_1}(qtz_2)^{r_1-\beta_1} D_{r_1}  = [z_1^{\beta_1}] \frac{D(z_1)}{1-qt\frac{z_2}{z_1}}$$ by \cref{lem:LargerPart}. By induction, one finds 
	$$\mcR_\beta^{(m+1)}(z_{m+2}) = \sum_{k\geq \beta_{m+1}}z_{m+2}^{k-\beta_{m+1}}[z_{m+1}^k]\mcR_\beta^{(m)}(z_{m+1})D(z_{m+1})=[z_{m+1}^{\beta_{m+1}}] \frac{\mcR_\beta^{(m)}(z_{m+1})D(z_{m+1})}{(1-qtz_{m+2}/z_{m+1})}$$ 
	for any $1\leq m\leq \ell-2$. Here, we used \cref{lem:LargerPart} with $\mcO(z)=\mcR_\beta^{(m)}(z_{m+1})D(z_{m+1}).$  In particular, for $m=\ell-2$,
	\[\mcR_{\beta}^{(\ell-1)}(z_{\ell}) = [z_1^{\beta_1} \dotsm z_{\ell-1}^{\beta_{\ell-1}}] \frac{D(z_1) \dotsm D(z_{\ell-1})}{\prod_{i=1}^{\ell-1} \bigl(1-qt z_{i+1}/z_i\bigr)}.\]
	Finally, we get
	\[\mcR_\beta = [z_\ell^{\beta_\ell}] \mcR_\beta^{(\ell-1)}(z_\ell) D(z_\ell) = [z_1^{\beta_1} \dotsm z_{\ell}^{\beta_{\ell}}] \frac{D(z_1) \dotsm D(z_{\ell})}{\prod_{i=1}^{\ell-1} \bigl(1-qt z_{i+1}/z_i\bigr)},\]
	where the first equality follows from \cref{prop:Qalpha_D}.	This finishes the proof of the theorem.
\end{proof}

\subsection{Concatenation of paths}
We now give a VOA formula for a composition of path operators. 
Fix an integer $m\geq 1$, and $\ell_1, \dotsc, \ell_m \geq 0$ and consider a concatenation of $m$ sequences of respective lengths $(\ell_1,\ell_2,\dots,\ell_m)$,
$$\bfbeta=(\beta^{(1)},\beta^{(2)},\dots,\beta^{(m)})\in \ZZ^{\ell_1}\times\dotsb\times \ZZ^{\ell_m}.$$ 
Such concatenation can be written
$\bfbeta=(\beta_1,\dots,\beta_\ell),$
where $\ell:=\ell_1+\dots+\ell_m$. We call \emph{the total size}  of $\bfbeta$ the integer $\sum_{1\leq i\leq \ell}\beta_i$.

We define 
\[\bR_{\bfbeta}:=\bR_{\beta^{(1)}}\times\dots\times\bR_{\beta^{(m)}}\]
the set of sequences of $m$ paths $(\gamma^{(1)},\dotsc,\gamma^{(m)})$ with $\gamma^{(i)}\in\bR_{\beta^{(i)}}$. Similarly, we define the composition of operators
$$\mcR_{\bfbeta}:=\mcR_{\beta^{(1)}}\dotsm\mcR_{\beta^{(m)}}=\sum_{(\gamma^{(1)},\dotsc,\gamma^{(m)})\in \bR_{\bfbeta}} \mcO_{\beta^{(1)}}(\gamma^{(1)})\dotsm\mcO_{\beta^{(m)}}(\gamma^{(m)})$$
and the valley increments associated to $\bfbeta$
$$\VI_{\bfbeta}:=\VI_{\beta^{(1)}}\times\dots\times\VI_{\beta^{(m)}},$$
and we extend the valley weight defined in \eqref{eq:valley_weight_PMAP} by multiplicativity.

We have the following immediate corollary from \Cref{thm:R_vertex_op}.
\begin{cor}\label{cor:R_vertex_op}
	Fix $\bfbeta=(\beta^{(1)},\beta^{(2)},\dots,\beta^{(m)})=(\beta_1,\dots,\beta_\ell)$, and let $L_k := \sum_{i=1}^k \ell_i$ with the convention $L_0 = 0$. Then
	\begin{equation}\label{eq:R_vertex_op}
		\mcR_{\bfbeta} :=\mcR_{\beta^{(1)}}\dotsm\mcR_{\beta^{(m)}}= [z_1^{\beta_1} \dotsm z_{\ell}^{\beta_{\ell}}] \frac{D(z_1) \dotsm D(z_{\ell})}{\prod_{k=1}^{m}\prod_{i=L_{k-1}+1}^{L_k-1} \bigl(1-qt z_{i+1}/z_i\bigr)}.    	\end{equation}
\end{cor}

\subsection{Normal ordering and non-consecutive indices}
\label{subsec:NormalOrdering}
Recall that $D(z) = \mcP_{-z} \mcT_{M/z}$ where $\mcP_{-z}$ acts by multiplication and $\mcT_{M/z}$ by translation. Normal ordering a product of $D$-operators consists in rewriting it such that all the translation operators are on the right of the multiplication operators. The elementary relation for normal ordering follows from \eqref{eq:mcT-mcP'} and is given by
\begin{equation*}
	\label{eq:NormalOrd}
	\mcT_{M/z} \mcP_{-w} = \omega(w/z) \mcP_{-w} \mcT_{M/z},
\end{equation*}
with $\omega(x)$ being the following series in $x$, 
\[\omega(x) = \Omega[-Mx] = \frac{(1-x)(1-qtx)}{(1-qx)(1-tx)}.\]
This gives
\begin{equation} \label{NormalOrdering}
	D(z) D(w) = \omega(w/z) \mcP_{-w-z} \mcT_{M/z+M/w}.
\end{equation}

Before plugging the above formula into \Cref{cor:R_vertex_op}, we define the set of \emph{non-consecutive indices} $\mcNI_{\bfbeta}$ of $\bfbeta$ as the set of pairs $(i,j)$ with $i<j$ that are not consecutive as integers, i.e. $j>i+1$, or that belong to two different sequences $\beta^{(n_i)}$ and $\beta^{(n_j)}$ in $\bfbeta$. More formally: 
\begin{multline}
	\mcNI_{\bfbeta}:=
	\left\{(i,j)|1\leq i<j\leq \ell, \text{ such that }i<j-1 \right.\\
	\left.\text{ or } (i=j-1 \text{ and }i\leq L_k<j \text{ for some }1\leq k\leq m)\right\},  
\end{multline}
where we denote the partial sums $L_k=\sum_{i=1}^k \ell_i$. In other terms, $\mcNI_{\bfbeta}$ is defined so that \eqref{eq:R_vertex_op} can be rewritten as
\begin{equation}
	\label{eq:R_vertex_opPrim}
	\mcR_{\bfbeta} = [z_1^{\beta_1} \dotsm z_{\ell}^{\beta_{\ell}}] \frac{D(z_1) \dotsm D(z_{\ell})}{\prod_{\substack{1 \leq i < j \leq \ell\colon\\ (i,j) \notin \mcNI_{\bfbeta}}} \bigl(1-qt z_{j}/z_i\bigr)}.
\end{equation}
Note that $\mcNI_{\bfbeta}$ depends only on the sequence $(\ell_1,\ell_2,\dots,\ell_m)$ and not $\bfbeta$ itself.
\begin{exe}
	If $\bfbeta=((2,0,0),(3))$, then $\mcNI_{\bfbeta}=\left\{\right(1,3),(1,4),(2,4),(3,4)\}.$
\end{exe}

\begin{thm}\label{thm:explicit_formula2}
	We set $\bM:=q+t-1$. Fix $\bfbeta=(\beta^{(1)},\beta^{(2)},\dots,\beta^{(m)})=(\beta_1,\dots,\beta_\ell)$. Then 
	\begin{equation}\label{eq:explicit_formula}
		\mcR_{\bfbeta}\cdot1 = \left[z_1^{\beta_1}\dotsm z_\ell^{\beta_\ell}\right]\Omega\left[\bM\sum_{1\leq i< j\leq \ell}z_j/z_i-qt\sum_{(i,j)\in\mcNI_{\bfbeta}}z_j/z_i-\sum_{1\leq i\leq \ell}z_iX\right].
	\end{equation}
\end{thm}

\begin{proof}
	We use \eqref{NormalOrdering} to normal order $D(z_1) \dotsm D(z_{\ell})$ in \cref{eq:R_vertex_op},
	\begin{equation*}
		D(z_1) \dotsm D(z_{\ell}) = \biggl(\prod_{1\leq i<j\leq \ell} \omega(z_j/z_i)\biggr) \mcP_{-z_1-\dotsb-z_\ell}\mcT_{M/z_1+\dotsb+M/z_\ell}. 
	\end{equation*}
	Applying this equation on 1, we get
	\begin{equation*}
		D(z_1) \dotsm D(z_{\ell})\cdot 1 = \biggl(\prod_{1\leq i<j\leq \ell} \omega(z_j/z_i)\biggr) \Omega\left[-\sum_{1\leq i\leq \ell}z_i X\right]
	\end{equation*}
	since $\mcT_x \cdot 1 = 1$ and $\mcP_x \cdot 1 = \Omega[x]$.
Plugging it into \eqref{eq:R_vertex_opPrim}, we get
	\begin{equation*}
		\mcR_{\bfbeta}\cdot 1 = [z_1^{\beta_1} \dotsm z_{\ell}^{\beta_{\ell}}] \prod_{1\leq i<j\leq \ell} \frac{(1-z_j/z_i) }{(1-qz_j/z_i)(1-tz_j/z_i)} \prod_{(i,j)\in\mcNI_\beta} (1-qtz_j/z_i) \Omega\left[-\sum_{1\leq i\leq \ell}z_i X\right]
	\end{equation*}
	which is equivalent to the plethystic exponentials in the Theorem.
\end{proof}

We deduce from \cref{thm:explicit_formula2} the expansion for the family $\mcR_{\bfbeta}\cdot 1$ in the Schur basis, the elementary basis, and the monomial basis.
\begin{cor}\label{cor:explicit:formula}
	Fix $\bfbeta=(\beta^{(1)},\beta^{(2)},\dots,\beta^{(m)})=(\beta_1,\dots,\beta_\ell)$ of total size $n\geq 0$. For any partition $\lambda$ of size $n$ and for the alphabet $Z =
    z_1+\dotsm+z_\ell$, we have
	\begin{align*}
		(-1)^n [s_\lambda[X]] \mcR_{\bfbeta}\cdot 1
		=&\left[z^{\bfbeta}\right]
		s_{\lambda'}[Z]\prod_{1\leq i<j\leq \ell}
		\frac{1-z_j/z_i}{(1-qz_j/z_i)(1-tz_j/z_i)}\prod_{(i,j)\in\mcNI_{\bfbeta}}
		(1-qtz_j/z_i),
	\end{align*}
	and 
	\begin{align*}
		(-1
		)^n[m_\lambda[X]] \mcR_{\bfbeta}\cdot 1
		=&\left[z^{\bfbeta}\right]
		e_{\lambda}[Z]\prod_{1\leq i<j\leq \ell}
		\frac{1-z_j/z_i}{(1-qz_j/z_i)(1-tz_j/z_i)}\prod_{(i,j)\in\mcNI_{\bfbeta}}
		(1-qtz_j/z_i),  
	\end{align*}
	and
	\begin{align*}
	 (-1)^n[e_\lambda[X]] \mcR_{\bfbeta}\cdot 1
	=&\left[z^{\bfbeta}\right]
	m_\lambda[Z]\prod_{1\leq i<j\leq \ell}
	\frac{1-z_j/z_i}{(1-qz_j/z_i)(1-tz_j/z_i)}\prod_{(i,j)\in\mcNI_{\bfbeta}}
	(1-qtz_j/z_i).
	\end{align*}
\end{cor}
\begin{proof}
	Each formula is a consequence of \cref{eq:explicit_formula} and one of the dual Cauchy formulas:
	\begin{align*}
		\Omega\left[-ZX\right]=\sum_{\lambda\in\YY}(-1)^{|\lambda|}s_{\lambda'}[Z]s_\lambda[X]=\sum_{\lambda\in\YY}(-1)^{|\lambda|}e_{\lambda}[Z]m_\lambda[X]=\sum_{\lambda\in\YY}(-1)^{|\lambda|}m_\lambda[Z]e_{\lambda}[X].
	\end{align*}
	where $\lambda'$ is the transpose of $\lambda$.
\end{proof}
Note that \cref{thm:explicit_formula} is then obtained from \cref{cor:explicit:formula} by taking $m=1$.

\section{Functional equations for the \texorpdfstring{$(q,t)$}{(q,t)}-tau function}
\label{sec:fun_eq}

\subsection{Preliminaries}
\label{sub:MacdonaldPrel}

Recall (see~\cite{GarsiaHaimanTesler1999}) that a ring of symmetric functions $\Lambda$ on the alphabet $X$ possesses a distinguished basis of the so-called modified Macdonald polynomials $\tH_\lambda[X]$. We will need two well-known properties of modified Macdonald polynomials.

\begin{thm}
	\label{thm:Macdonald}
	We have the following:
	\begin{enumerate}
		\item[(i)] \textbf{Pieri Rule \cite{Macdonald1995}:} There exist coefficients $d^{\mu,\lambda}\in \QQ(q,t)$ such that
		\begin{equation*}
			-e_1[X]\tH_\mu=\sum_{\mu\nearrow\lambda}d^{\mu,\lambda}\tH_\lambda
		\end{equation*}
		where the sum is taken over partitions $\lambda$ obtained from $\mu$ by adding exactly one cell.
		\item[(ii)] \textbf{Macdonald operator $D_0$ \cite[Theorem~1.2]{GarsiaHaimanTesler1999}:} The modified Macdonald polynomials are eigenfunctions of the operator $D_0 = [z^0] D(z)$ defined in~\eqref{eq:Dk}
		\begin{equation*}
			D_0\cdot \tH_{\lambda}=\biggl(1-M\sum_{(i,j)\in\lambda} q^{j-1} t^{i-1}\biggr)\tH_{\lambda},\quad \text{for any $\lambda\in \YY$,}
		\end{equation*}
        where the sum is taken over all cells $(i,j)$ of the Young diagram $\lambda$.
	\end{enumerate}
\end{thm}

Recall that the Macdonald Cauchy kernel can be written as follows
\begin{equation*}
    \Omega\left[-\frac{XY}{M}\right]=\exp\left(\sum_{i \geq 1} \frac{p_i[X]p_i[Y]}{i p_i[-M]}\right) = \sum_{\mu\in\YY}\frac{p_{\mu}[X]p_{\mu}[Y]}{z_\mu p_\mu[-M]},
\end{equation*}
which motivates the definition of the \emph{star scalar product} $\langle .,.\rangle_*$ as the deformation of the Hall scalar product given by 
$$\langle p_\mu,p_\nu\rangle_*=z_\mu\delta_{\mu,\nu}p_\mu[-M]=(-1)^{\ell(\mu)}z_\mu\delta_{\mu,\nu}\prod_{1\leq i\leq \ell(\mu)}\left(1-q^{\mu_i}\right)\left(1-t^{\mu_i}\right), \quad \text{for any $\mu,\nu\in\YY$.} $$
Modified Macdonald polynomials are orthogonal with respect to the star scalar product:
\begin{equation}\label{eq:orthogonality_Macdonald}
  \langle \tH_\lambda,\tH_\rho\rangle_*=\delta_{\lambda,\rho}\normH{\lambda},  
\end{equation}
we refer to \cite[Chapter VI]{Macdonald1995} for a combinatorial formula of the squared norm $\normH{\lambda}$. The following Cauchy formula then holds:
\begin{equation}\label{eq:Cauchy}
    \Omega\left[-\frac{XY}{M}\right]=\sum_{\lambda\in\YY}\frac{\tH_\lambda[X]\tH_\lambda[Y]}{\normH{\lambda}}.
\end{equation}

Recall that we denote $\Lambda^\KK_A$ the algebra of symmetric functions in the alphabet $A = X,Y$ with coefficients in $\KK :=\QQ(q,t)\llbracket U,V \rrbracket$, where $U = u_1,u_2,\dotsc$ and $V = v_1,v_2,\dotsc$ are two infinite sequences of variables. Then the $G$-weighted $(q,t)$-tau function is the series in $\Lambda^\KK_X\otimes\Lambda^\KK_Y\llbracket z\rrbracket$ defined by:
\begin{equation*}
	\tau_G(z,X,Y) :
	= \sum_{\lambda\in\mathbb
		Y}z^{|\lambda|}\frac{\tH_\lambda[X]\tH_\lambda[Y]}{\normH{\lambda}}\prod_{(i,j)\in\lambda} G(q^{j-1} t^{i-1}),
\end{equation*}
where 
\[ G_1(\hbar) := 1+\sum_{n=1}^\infty u_n \hbar^n \in \QQ[U,V]\llbracket \hbar \rrbracket,\quad  G_2(\hbar) :=  1+\sum_{n=1}^\infty v_n \hbar^n \in \QQ[U,V]\llbracket \hbar \rrbracket,\]
and $G(q^{j} t^{i}) := \frac{G_1(q^{j} t^{i})}{G_2(q^{j} t^{i})} \in \KK$.

There is an equivalent way of defining the $G$-weighted $(q,t)$-tau function, which is closer in spirit to the setting used in~\cite{GarsiaHaimanTesler1999}. Let $G(\hbar) \in \KK \llbracket \hbar \rrbracket$ be a series such that the substitution $G(q^j t^i)$ is a well-defined element of $\KK$ for all $j,i \geq 0$.
  
For an alphabet $A = X,Y$ we define the linear operator $\Pi^A_G \colon \Lambda_A^\KK \rightarrow \Lambda_A^\KK$ by its action on the basis of Macdonald polynomials
\begin{equation*}\label{eq:Pi_Operator}
  \Pi^A_G\cdot \tH_\lambda[A] := \prod_{(i,j)\in\lambda} G(q^{j-1} t^{i-1}) \tH_\lambda[A],\qquad \forall \lambda\in\YY.
\end{equation*}
Then, it is immediate from the Cauchy identity~\cref{eq:Cauchy} that
\[ \tau_G(z,X,Y) = \Pi^X_G\Omega\left[-\frac{zXY}{M}\right] = \Pi^Y_G\Omega\left[-\frac{zXY}{M}\right].\]

\subsection{Functional equations}
Throughout the paper, the series $F$ will be $F=G_1$ or $F=G_2$, with expansion $F(\hbar)=\sum_{i \geq 0}a_i \hbar^i$ ($a_i=u_i$ in the first case and $a_i=v_i$ in the second one and with the convention $a_0 = 1$). We will use
\begin{equation}
	\label{eq:A_F}
	\A_F \coloneqq \sum_{i \geq 0}a_i \ad^i_{\frac{D_0}{-M}}(-e_1[X]), \quad \text{with $\ad_{A}(B) := [A,B] = AB-BA$}
\end{equation}
for any operators $A$ and $B$. Then define the family of operators
\begin{equation}
	\label{def:AFell}
	\A^{(1)}_F := \ad_{\frac{D_0}{-M}} \left(\A_F\right),\qquad \A^{(\ell)}_F =
	\ad^{\ell-1}_{\A_F/M}(\A^{(1)}_F),\quad \text{for $\ell \geq 2$}.
\end{equation}

\begin{rmk}
Note that we have already used the notation $\A^{(\ell)}_F$ in \cref{thm:dif_eq'} for operators that were defined in terms of path operators. This is not a conflict in notation, as we will prove in \cref{sec:path_op} that the operators $\A^{(\ell)}_F$ defined by \eqref{def:AFell} can be expressed in terms of the path operators as stated in~\cref{thm:dif_eq'}. For now, however, we will only use the definition~\eqref{def:AFell}. 
\end{rmk}

\begin{rmk}\label{rmk:diff_op}
Recall that $\Lambda^\KK$ is isomorphic to the polynomial ring $\KK[z_1,z_2,\dots]$, where $z_1,z_2,\dots$ are generators of $\Lambda$; for instance $e_1,e_2,\dots$, $h_1,h_2,\dots$, or $p_1,p_2,\dots$. Then $\A_{F}^{(\ell)}$ are differential operators in the sense that
\[ \A_{F}^{(\ell)} \colon \KK[z_1,z_2,\dots] \rightarrow \KK[z_1,z_2,\dots],\quad \A_{F}^{(\ell)} = \sum_{\alpha,\beta} a_{\alpha}^\beta z^\alpha \partial_{z^\beta},\quad a_{\alpha}^\beta \in \KK,\]
where we sum over multi-indices $\alpha,\beta$ with finitely many non-zero terms, and $z^\alpha := z_1^{\alpha_1} z_2^{\alpha_2} \dotsc$, $\partial_{z^\beta} := \partial_{z_1}^{\beta_1} \partial_{z_2}^{\beta_2} \dotsc$. This can be checked using the expansion of $D_0 = \sum_{n\geq 0} (-1)^n e_n[X] h_n^\perp[MX]$ and expanding $h_n^\perp$ in terms of $p_k^\perp=\frac{k\partial}{\partial p_k}$.
\end{rmk}

We start with the following functional equation (or PDE by \cref{rmk:diff_op}) satisfied by the $(q,t)$-tau function.
\begin{thm}\label{thm:dif_eq}
	For any $\ell\geq 1$ we have
	\begin{equation}\label{eq:diff_eq2}
		z^\ell\A^{(\ell)}_{G_1}(X)\cdot\tau_G(z,X,Y)=\left(\A^{(\ell)}_{G_2}(Y)\right)^*\cdot\tau_G(z,X,Y).  
	\end{equation}
where $\left(\A_{G_2}^{(\ell)}\right)^*$ is the adjoint of $\A_{G_2}^{(\ell)}$ with respect to the star scalar product.
\end{thm}

In order to prove \cref{thm:dif_eq}, we need to find how the operators $\A^{(\ell)}_F$ act on modified Macdonald polynomials, starting with $\ad^i_{D_0/M}(-e_1[X])$, which is possible by~\cref{thm:Macdonald}. Remarkably, no explicit expression of the coefficients $d^{\mu,\lambda}$ appearing in the Pieri rule (\cref{thm:Macdonald}(i)) is required in our work, only the fact that they vanish if $\lambda$ is not obtained from $\mu$ by the addition of one cell.

For a skew diagram $\lambda/\mu$, we use the notation $F(\lambda/\mu):=\prod_{(i,j)\in\lambda/\mu}F(q^{j-1}t^{i-1}).$

\begin{lem}\label{lem:DeltaG_A+}
	For any $\ell\geq 1$, we have	\[\Pi_G\cdot \A^{(\ell)}_{G_2}\cdot \Pi_G^{-1}=\A^{(\ell)}_{G_1}.\] 
\end{lem}
\begin{proof}
	From \cref{thm:Macdonald} and the defining formula~\eqref{eq:A_F} for $\A_F$ we find
	\begin{equation*} \label{eq:A_PonMacdonald}
		\A_F \tH_\mu = \sum_{\mu\nearrow\lambda} d^{\mu,\lambda} F(\lambda/\mu) \tH_\lambda.
	\end{equation*}
With the definition of $\Pi_G$, this gives 
\begin{equation}\label{eq:A_PI}
    \Pi_G \A_{G_2}\Pi^{-1}_G=\A_{G_1}.
\end{equation}
Commuting \cref{eq:A_PI} with $\frac{-D_0}{M}$ gives the lemma for $\ell=1$. Then for $\ell\geq 2$, one has $\A^{(\ell)}_F = \frac{1}{M}[\A_F, \A_F^{(\ell-1)}]$, so the lemma is obtained from \cref{eq:A_PI} by induction.
\end{proof}

\noindent We obtain \Cref{thm:dif_eq} by combining \cref{lem:DeltaG_A+} and the following lemma.
\begin{lem}\label{lem:diff_eq}
	Let $\mcO_1$ and $\mcO_2$ be two operators on $\Lambda^\KK$ which have the same homogeneous degree $\ell\geq0$. Then
	\begin{equation*}\label{eq:diff_DeltaG}
		z^\ell\mcO_2(X)\cdot \tau_G(z,X,Y)=\mcO_1^{*}(Y)\cdot \tau_G (z,X,Y)\ \Longleftrightarrow\ \Pi_G\cdot \mcO_1\cdot \Pi_G^{-1}=\mcO_2.  
	\end{equation*}
\end{lem}

\begin{proof}
		Since Macdonald polynomials are a basis of $\Lambda^\KK$, we have the equivalence
		\begin{align*}
			z^\ell\mcO_2(X)&\cdot \tau_G(z,X,Y)=\mcO_1^{*}(Y)\cdot \tau_G (z,X,Y)\Longleftrightarrow \\
			 &\forall \lambda, \quad \langle z^\ell \mcO_2(X)\cdot \tau_G(z,X,Y), \tH_\lambda(Y)\rangle_*=\langle \tau_G(z,X,Y),\mcO_1(Y)\cdot \tH_\lambda(Y)\rangle_*.
		\end{align*}
        		Here the scalar product is taken with respect to the variables $Y$.
		For $1\leq i \leq 2$, define the coefficients $a^{(i)}_{\lambda,\xi}$ by
	$$\mcO_i\cdot \tH_\lambda=\sum_{\xi\vdash |\lambda|+\ell}a^{(i)}_{\lambda,\xi} \tH_\xi.$$
Using the orthogonality of Macdonald polynomials \cref{eq:orthogonality_Macdonald}, we get
		\begin{align*}
			 z^\ell&\mcO_2(X)\cdot \tau_G(z,X,Y)=\mcO_1^{*}(Y)\cdot \tau_G (z,X,Y) 
			 \Longleftrightarrow \\
			 &\forall \lambda,\quad G(\lambda)\sum_{\xi\vdash
			   |\lambda|+\ell}a^{(2)}_{\lambda,\xi}\, \tH_\xi =  \sum_{\xi\vdash
			   |\lambda|+\ell}G(\xi) a^{(1)}_{\lambda,\xi}\, \tH_\xi \Longleftrightarrow\\
			 &\forall \lambda,\quad \sum_{\xi\vdash
			   |\lambda|+\ell}a^{(2)}_{\lambda,\xi}\, \tH_\xi =  \sum_{\xi\vdash
			   |\lambda|+\ell}G(\xi) a^{(1)}_{\lambda,\xi} G^{-1}(\lambda) \, \tH_\xi \Longleftrightarrow\\
		  &\forall \lambda,\quad \mcO_2\,
			\tH_\lambda = \Pi_G\cdot \mcO_1\cdot \Pi_G^{-1}\,\tH_\lambda, 
		\end{align*}
	which finishes the proof.
	\end{proof}

\subsection{Uniqueness}

\cref{thm:dif_eq} implies that if $\tau(z) \in \Lambda_X^\KK\otimes\Lambda_Y^\KK\llbracket z \rrbracket$ is equal to $\tau_G(z,X,Y)$, then $\tau(z)$ is a solution of a system of PDEs. We prove that the converse is also true, i.e. $\tau(z)$ is uniquely determined by the functional equations.
\begin{thm}\label{thm:uniqueness}
	Let $\tau(z) \in \Lambda_X^\KK\otimes\Lambda_Y^\KK\llbracket z \rrbracket$ be a function which satisfies the functional equations of \cref{thm:dif_eq}, and $\tau(z) = 1 + O(z)$. Then $\tau(z) = \tau_G(z,X,Y)$.
\end{thm}

It turns out that we can construct a basis of $\Lambda^\KK$ by repeatedly acting on $1$ with the operators $\A_F^{(\ell)}$. This construction leads to \cref{thm:uniqueness}.
We define for $\lambda\in \YY$
\begin{equation*}\label{eq:def_aaa}
  \aaa_{F,\lambda}:=\A_F^{(\lambda_1)}\cdot \A_F^{(\lambda_2)}\dotsm \A_F^{(\lambda_{\ell(\lambda)})}\cdot 1
\end{equation*}
which is a symmetric function of degree $|\lambda|$.
\begin{rmk}
    Since the operators $(\A_F^{(\ell)})_{\ell\geq 1}$ do not commute in general, taking a different order on the parts of $\lambda$ in the previous definition would lead to a different definition of the functions $\aaa_{F,\lambda}$.
\end{rmk}
\begin{prop}\label{prop:basis_a_lambda}
    For $F\in\left\{G_1,G_2\right\}$, the family $(\aaa_{F,\lambda})_{\lambda\in\YY}$ is a basis of $\Lambda^\KK$.
\end{prop}
Before proving this proposition, let us explain how it implies~\cref{thm:uniqueness}. 

For any partition $\lambda$, the differential equations \cref{eq:diff_eq2} allow us to write
\begin{multline*}
  z^{|\lambda|}\A_{G_1}^{(\lambda_1)}(X)\cdot \A_{G_1}^{(\lambda_2)}(X)\cdots \A_{G_1}^{(\lambda_{\ell(\lambda)})}(X)\cdot \tau(z)\\
  =\left(\A_{G_2}^{(\lambda_{\ell(\lambda)})}(Y)\right)^*\cdots  \left(\A_{G_2}^{(\lambda_2)}(Y)\right)^* \cdot\left(\A_{G_2}^{(\lambda_1)}(Y)\right)^*\cdot\tau(z).  
\end{multline*}
 We now extract the coefficient of $z^{|\lambda|}p_\emptyset[Y] =
 z^{|\lambda|}$. On the left-hand side we get
 \begin{align*}
    \A_{G_1}^{(\lambda_1)}(X)\cdots
   \A_{G_1}^{(\lambda_{\ell(\lambda)})}(X)\cdot\left[z^0 \right]\tau(z)
    &=\A_{G_1}^{(\lambda_1)}(X)\cdots \A_{G_1}^{(\lambda_{\ell(\lambda)})}(X)\cdot 1\\
    &=\aaa_{G_1,\lambda}(X).
 \end{align*}
 On the right-hand side, we obtain 
 \begin{align*}
   \left[z^{|\lambda|}p_\emptyset[Y] \right]&\left(\A_{G_2}^{(\lambda_{\ell(\lambda)})}(Y)\right)^*\cdots  \left(\A_{G_2}^{(\lambda_2)}(Y)\right)^* \cdot\left(\A_{G_2}^{(\lambda_1)}(Y)\right)^*\cdot \tau(z)\\
   &=\left\langle p_\emptyset[Y], \left[z^{|\lambda|}\right]\left(\A_{G_2}^{(\lambda_{\ell(\lambda)})}(Y)\right)^*\cdots  \left(\A_{G_2}^{(\lambda_2)}(Y)\right)^* \cdot\left(\A_{G_2}^{(\lambda_1)}(Y)\right)^*\cdot \tau(z)\right\rangle_*\\
   &=\left\langle \A_{G_2}^{(\lambda_1)}(Y)\cdot
   \A_{G_2}^{(\lambda_2)}(Y)\cdots \A_{G_2}^{(\lambda_{\ell(\lambda)})}(Y)\cdot 1, \left[z^{|\lambda|}\right]\tau(z)\right\rangle_*\\
   &=\left\langle \aaa_{G_2,\lambda}(Y), \left[z^{|\lambda|}\right]\tau(z)\right\rangle_*,
 \end{align*}
 where the scalar product is taken with respect to the alphabet $Y$. Let $\bbb_{F,\lambda}$ denote the dual basis of $\aaa_{F,\lambda}$, i.e.
 $$\left\langle\aaa_{F,\lambda}, \bbb_{F,\mu}\right\rangle_*=\delta_{\lambda,\mu}, \quad \text{for any $\lambda$ and $\mu$.}$$
Therefore, the differential equations imply that for any $\lambda$
\begin{equation*}
    \aaa_{G_1,\lambda}(X)=\left[z^{|\lambda|}\bbb_{G_2,\lambda}(Y)\right]\tau(z).
  \end{equation*}
  In other terms,
\begin{equation*}\label{eq:tau_G_formula}
	\tau(z)=\sum_{\lambda\in\YY}z^{|\lambda|}\aaa_{G_1,\lambda}(X)\bbb_{G_2,\lambda}(Y).  
\end{equation*}
This proves that $\tau(z)$ is completely determined  by the equations \cref{eq:diff_eq'}, and therefore $\tau(z)=\tau_G(z,X,Y)$.
\begin{rmk}
    Note that when $G_1=G_2$, the function $\tau_G(z,X,Y)$ corresponds to the Cauchy kernel (see \cref{eq:Cauchy}), and the previous equation follows from the fact that $\aaa_{G_1}$ and $\bbb_{G_1}$ are dual with respect to the star scalar product.
\end{rmk}

In order to conclude the proof of \cref{thm:uniqueness}, we need to prove \cref{prop:basis_a_lambda}. The proof of the latter is a direct consequence of the combinatorial expression for $\A_F^{\ell}$ is terms of the path operators, which we prove in \cref{sec:path_op}. This expression, together with \cref{thm:dif_eq} and \cref{thm:uniqueness} gives us our main \cref{thm:dif_eq'}.

\subsection{Rewriting in terms of path operators}
\label{sub:path_op}

Recall that for any $\beta \in \ZZ^\ell$ we denote $\bR_\beta$ the set of alternating paths that start and end at the same points as $\gamma_\beta$ (i.e. have the same length and degree) and that stay weakly above $\gamma_\beta$. 

From now on, we will be working with sequences $\beta \in \ZZ_{>0}\times \ZZ_{\geq 0}^{\ell-1}$ and we will use an alternative description of path operators as \emph{paths decorated with particles}: we think of a path $\gamma\in \bR_\beta$ as an alternating path starting at $(0,0)$ ending at $(2\ell,\sum_{1\leq i\leq \ell}\beta_i)$ and such that:
\begin{itemize}
    \item  the $i$-th peak is decorated with $\beta_i$ particles, 
    \item for each valley $V$ of $\gamma$, the $y$-coordinate is at least the number of particles sitting on peaks on the left of $V$. The $\beta$-height of $V$ (as defined in \cref{eq:def_height}) is given by the difference between these two integers.
\end{itemize}
Notice that the condition $\beta \in \ZZ_{>0}\times \ZZ_{\geq 0}^{\ell-1}$ implies that there is always  at least one particle on the first peak.

We are now ready to give our second parametrization of path operators. If $\beta$ is a sequence of size $n$, we define $\alpha := \psi(\beta) \in \ZZ_{\geq 0}^n$ as the sequence consisting of the distances between consecutive particles of a path $\gamma\in\bR_\beta$, where the distance is interpreted as the number of peaks between the particles. By convention, $\alpha_n$ is the number of peaks after the last particle. We refer to~\cref{fig:RtoQ} for an example.

Equivalently, $\alpha=\psi(\beta)$ is the sequence obtained from $\beta$ as follows: we start from the empty sequence $\alpha=\emptyset$, and we read the sequence $\beta$ from left to right; each time we read an integer $i$ we add a sequence of $i$ zeroes, and we increase the right-most integer of $\alpha$ by 1, except at the end of the procedure -- in which case we only add a sequence of $i$ zeroes. We refer to \cref{fig:RtoQ} for an example.

\begin{figure}[h]
	\centering
	\includegraphics[width = \linewidth]{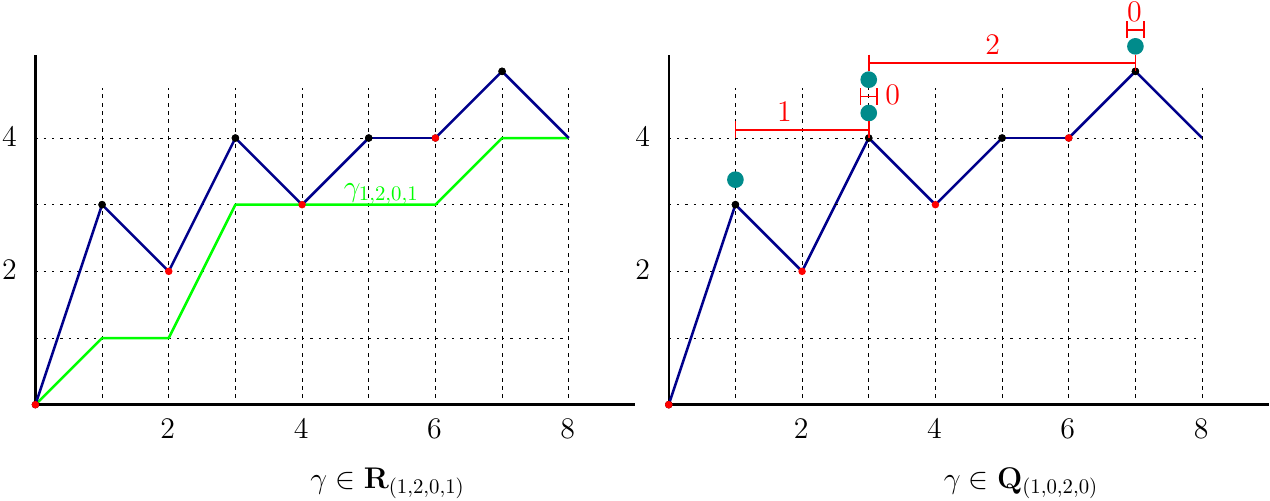}
	\caption{A path $\gamma$ belonging to $\bR_{(1,2,0,1)}$ with the associated path $\gamma_\beta$ on the left. On the right we visualize vector $\beta = (1,2,0,1)$ by placing particles over the peaks, and we reparametrize the set $\bR_{(1,2,0,1)}$ as $\bQ_{(1,0,2,0)}$ by interpreting $(1,0,2,0) = \psi(1,2,0,1)$ as the distances between the particles.}
	\label{fig:RtoQ}
\end{figure}

This defines a bijection $\psi$ between sequences in $\ZZ_{>0}\times \ZZ_{\geq 0}^{\ell-1}$ with size $n$ and sequences in $\ZZ_{\geq 0}^{n}$ with size $\ell-1$. Therefore, for any $\alpha \in \ZZ_{\geq 0}^{n}$ we can reparametrize the set $\bR_\beta$ as $\bQ_\alpha$, where $\beta := \psi^{-1}(\alpha)$, and similarly
\begin{equation}\label{eq:Q^alpha-Q_alpha}
	\mcQ_{\alpha}:=\mcR_{\psi^{-1}(\alpha)}=\sum_{\bfgamma\in\bR_{\psi^{-1}(\alpha)}}\mcO_{\psi^{-1}(\alpha)}(\gamma).  
  \end{equation}

\begin{exe}\label{exe:R_k}
For any $r\geq 0$, we have $\mcQ_{(r)}=\mcR_{(1,0^r)}$ and $\mcQ_{(0^r)}=\mcR_{(r)}$.    
\end{exe}

Our goal is to show that for $F(\hbar) = \sum_{i\geq 0} a_i \hbar^i \in \{G_1, G_2\}$, one has the following expression for the operator $\A_F^{(\ell)}$ (as defined on \cref{def:AFell}):
\begin{thm}\label{thm:dif_eq''}
	For any $\ell\geq 1$, we have
	\begin{equation}\label{eq:diff_eq'}
		\A_F^{(\ell)} = \sum_{\alpha} a_\alpha \mcQ_\alpha,
	\end{equation}
	where the sum is taken over $\alpha \in \ZZ_{\geq 0}^{\ell}$, and $a_\alpha := \prod_{i=1}^\ell a_{\alpha_i}$.
\end{thm}
Informally, $\A_F^{(\ell)}$ corresponds then to the sum over all path operators with $\ell$ particles, with an extra weight $a_i$ for a distance $i\geq 0$ between two particles. 

We devote the whole \cref{sec:path_op} to proving this theorem. Before we present the proof, let us show how \cref{thm:dif_eq''} implies \cref{prop:basis_a_lambda}, and thus \cref{thm:uniqueness} and \cref{thm:dif_eq'}.

\begin{proof}[Proof of Proposition \ref{prop:basis_a_lambda}]

	We start by proving that $\aaa_{F,\lambda}$ is a basis when $q=t=1$. In this case, $M=0$ and the derivative parts of the path operators (defined in \cref{eq:def_step_operator}) become
	$$\mcO(m)_{\big|q=t=1}=\left\{\begin{array}{cc}
		 1&\text{if } m=0   \\
	   0&\text{if } m<0.
	\end{array}\right.$$
	As a consequence, only paths with no down steps contribute to the expansion of a path operator $\mcQ_\alpha$ \cref{eq:Q^alpha-Q_alpha}, and all paths operators becomes multiplicative. More precisely, for $\alpha\in\ZZ_{\geq 0}^\ell$ we have
	$${\mcQ_{\alpha}}_{\big|q=t=1}=(-1)^{\ell}e_{\ell}+\sum_{\substack{\mu \vdash \ell\\{\ell(\mu)>1}}}c_\mu e_\mu,$$
	for some coefficients $c_\mu$. Indeed, the term in $e_{\ell}$ corresponds to paths for which all steps are flat except for the first one which has size $\ell$. We deduce that the coefficient of this term in the operator $\A_{F}^{(\ell)}$ is given by
	\begin{equation}
	  (-1)^\ell[e_\ell]\A_{F}^{(\ell)}\cdot 1_{\big | q=t=1}=\sum_{\alpha\in\ZZ_{\geq 0}^\ell} a_\alpha=\left(\sum_{k\geq 0} a_k\right)^\ell.
	\end{equation}
	Note that this quantity is well defined and invertible in the space of formal power series $\KK$ since $a_0=1$ for $F\in\{G_1, G_2\}$.
	We deduce the expansion of $\aaa_\lambda$ in the elementary basis:
	$$(-1)^{|\lambda|}\aaa_{F,\lambda} {}_{\big | q=t=1}=e_\lambda\prod_{1\leq i\leq \ell(\lambda)}\sum_{\alpha \in \ZZ_{\geq 0}^{\lambda_i}} a_\alpha+\sum_{\mu<\lambda}d_{\mu}e_\mu,$$
	for some coefficients $d_\mu \in \KK$, and the rightmost sum is over partitions $\mu$ smaller than $\lambda$ with respect to the dominance order.
	By triangularity, we deduce that $\aaa_{F,\lambda}$ is a basis for $\Lambda$, when $q=t=1$.

	 Let $\mathcal{M}_n$ be the matrix of $(\aaa_{F,\lambda})_{\lambda\vdash n}$ in the elementary basis $(e_\mu)_{\mu\vdash n}$. Since the evaluation $\aaa_{F,\lambda}(q=t=1)$ is a basis, then $\det(\mathcal{M}_n)(q=t=1)$ is invertible in $\KK$. In particular,  
	$\det(\mathcal{M}_n)$ is also invertible in $\KK$, which finishes the proof.
	\end{proof}

\section{Building the path operators}
\label{sec:path_op}

In order to prove \cref{thm:dif_eq''}, we will study commutation relations between more and more complicated operators, starting from very simple ones. These commutation relations will have a natural interpretation in terms of extending simpler paths to more complicated ones.

The section is organized as follows: in \cref{ssec:path_prel} we give some preliminary results about $q$-integers and path operators. \cref{ssec:first_comm_rel} is devoted to commutation relations of operators with only one particle giving a combinatorial interpretation of the operators $\A_F$ (\cref{eq:A_F}). In \cref{ssec:second_comm_rel} we prove a family of commutation relations for symmetrized version of the operators $\mcQ_\alpha$ (see \cref{thm:P_commutation}).  Finally, we finish the proof of \cref{thm:dif_eq''} in \cref{ssec:proof_comb_formula}.
\subsection{Preliminaries}\label{ssec:path_prel}

We use the usual notation for $q$-integers;
$$[m]_q:=
\left\{
\begin{array}{cc}
    1+q+\dots+q^{m-1} & \text{if $m>0$}  \\
    -(1+q+\dots+q^{-m-1}) & \text{if $m<0$}\\
    0 & \text{if $m=0$}.
\end{array}\right.$$

\begin{lem}\label{lem:h_specialized}
  Let $n > 0$. We have
  \begin{align*}
    \frac{h_n\left[M\right]}{M} &= [n]_{qt},\\
    \frac{h_n\left[-M\right]}{-M} &= t^{n-1}[n]_{qt^{-1}} = q^{n-1}[n]_{tq^{-1}}. 
\end{align*}
  \end{lem}
  \begin{proof}
    This is an easy computation which relies on the following identities:
    \[ h_n[1+qt] = [n+1]_{qt}, \qquad h_n[q+t] = t^{n}[n+1]_{qt^{-1}} =
    q^{n}[n+1]_{tq^{-1}},\]
    for the homogeneous functions,
    \[e_n[1+qt] = \begin{cases} 1 &\text{ if } n = 0,\\
        1 +qt &\text{ if } n = 1,\\
        qt &\text{ if } n = 2,\\
      0 &\text{ if } n >2,\end{cases}, \qquad e_n[q+t] = \begin{cases} 1 &\text{ if } n = 0,\\
        q +t &\text{ if } n = 1,\\
        qt &\text{ if } n = 2,\\
        0 &\text{ if } n >2,\end{cases},\]
    for elementary functions, and the following recursions:
    \begin{align*}
        h_n\left[(1-t)(1-q)\right]
        &=\sum_{0\leq i\leq n}(-1)^{n-i}e_{n-i}\left[q+t\right]h_i\left[1+qt\right],\\
        h_n\left[-(1-t)(1-q)\right]
        &=\sum_{0\leq i\leq n}(-1)^{n-i}e_{n-i}\left[1+qt\right]h_i\left[q+t\right].\qedhere
    \end{align*}
\end{proof}

Recall the following commutation relation of \cref{eq:mcT-mcP'}:
 \begin{equation}\label{eq:mcT-mcP}
  \mcT_Y\cdot \mcP_Z=\Omega[YZ]\mcP_Z\cdot \mcT_Y,   
 \end{equation}
or equivalently
\begin{equation}\label{eq:mcT-mcP2}
\Omega[-YZ]\mcT_Y\cdot \mcP_Z=\mcP_Z\cdot \mcT_Y.
\end{equation}
Define $k \wedge r := \min(k,r), k \vee r := \max(k,r)$. We start from the following commutation relations for the one step operators $\mcO(k)$ defined in \cref{eq:def_step_operator}.

\begin{lem}\label{lem:com_h}
For any $k,r\geq 0$ we have
\begin{align}\label{eq:com_h2}
    \left[\mcO(-k),\mcO(r)\right]
        &=-M\sum_{1\leq i\leq k\wedge r} t^{i-1}[i]_{qt^{-1}}\cdot \mcO(r-i)\mcO(-(k-i))\\
        &=-M\sum_{1\leq i\leq k\wedge r} [i]_{qt}\cdot \mcO(-(k-i))\mcO(r-i).\label{eq:com_h}
    \end{align}
\end{lem}

\begin{proof}

    We take $Y=-z_1$ and $Z=M/z_2$ in \cref{eq:mcT-mcP} and \cref{eq:mcT-mcP2} respectively, and we extract the coefficient of $z_1^rz_2^{-k}$.
    We get
      \begin{align*}
    \left[h^\perp_k[MX],(-1)^r e_r[X]\right]
        =\sum_{1\leq i\leq k\wedge r} h_{i}[-M]\cdot (-1)^{r-i}e_{r-i}[X]h^\perp_{k-i}[MX]\\
        =-\sum_{1\leq i\leq k\wedge r} h_{i}[M]\cdot h^\perp_{k-i}[MX](-1)^{r-i}e_{r-i}[X].
    \end{align*}
    We conclude using the identities of \cref{lem:h_specialized} and the definition of the operators $\mcO(i)$ (\cref{eq:def_step_operator}).
\end{proof}

\subsection{First commutation relation}\label{ssec:first_comm_rel}
The main purpose of this section is to prove a first commutation relation, which gives a formula for the operator $\A_F$ (see \cref{eq:A_F}) in terms of path operators\footnote{We will give in \cref{prop:D0_P} an analog of this result for any operator $\mcQ_\alpha$.}.

 \begin{thm}[First commutation relation]\label{thm:A_n-P_n}
The following relations hold:
$$\frac{-1}{M}\left[D_{0},-e_1[X]\right] = \mcQ_{0},\qquad \frac{-1}{M}\left[D_{0},\mcQ_{n-1}\right] = \mcQ_{n},$$
for all $n \geq 1$. As a consequence,
    $$\A_F = -e_1[X]+\sum_{n \geq 0}a_{n+1}\mcQ_n.$$
\end{thm}

Without further mention, all paths in this section are one particle operators i.e paths decorated with only one particle on the first peak. In other terms, they belong to a set $\bR_{1,0^{n}}=\bQ_{n}$ for some $n\geq 0$.
 
The first step is to interpret combinatorially the commutation relations of \cref{lem:com_h}. We introduce the following definition.
\begin{defi}\label{def:path_extension}
    Consider an alternating path $\gamma=(\gamma_1,\gamma_2,\dots,\gamma_{2n})$. Let $1\leq j\leq 2n$ such that $r:=\gamma_j\neq 0$.
    
    We say that a path $\gamma'$ of length $2n+2$ is obtained by extending the $j$-th step of $\gamma$ and we write $\gamma\nearrow_j\gamma'$, if the step $\gamma_j$ is replaced by three steps $\gamma'_j,\gamma'_{j+1},\gamma'_{j+2}$ as follows.
    \begin{itemize}
        \item If $r>0$, then $(\gamma'_j,\gamma'_{j+1},\gamma'_{j+2})=(k,-(k-i),r-i)$, for some integers $k>0$ and $0<i\leq k\wedge r$.
        \item If $r<0$, then $(\gamma'_j,\gamma'_{j+1},\gamma'_{j+2})=(r+i,k-i,-k)$ for some
          integers $k>0$ and $0<i\leq k\wedge -r$.
    \end{itemize}
    These two cases are illustrated in \cref{fig:add_valley}. 
    The ordered pair of valleys involved in this extension (given from left to right) will be denoted $\Valley(\gamma\nearrow_j\gamma')$. In other terms,
    $\Valley(\gamma\nearrow_j\gamma')=(V,V')$
    where $V$ and $V'$ are the valleys of $\gamma'$ with respective $x$-coordinates $2\lfloor\frac{j}{2} \rfloor$ and $2\lfloor \frac{j}{2}\rfloor+2$). 

    When $\gamma_j=0$, the extension set $\gamma\nearrow_j\gamma'$ will be empty by convention.
\end{defi}
Note that when the $j$-th step of $\gamma$ is extended, the indices of the steps on its right increase by 2: $\gamma_l=\gamma'_{l+2}$ for any $i<l\leq 2n$. Moreover, if the step $\gamma_j$ is incident to a valley $V$, the extension operation creates a valley $\tilde V$ whose $y$-coordinate is larger than the $y$-coordinate of $V$.

\begin{exe}
    Consider the following alternating paths
    \begin{equation*}
        \left\{\begin{array}{l}
    \gamma=(3,-2,4,-2), \\
    \gamma'=(3,-2,6,-5,3,-2),\\
    \gamma''=(3,0,0,-2,4,-2).
        \end{array}\right.
    \end{equation*}
    
    Then $\gamma'$ is obtained form $\gamma$ by extending the third step of $\gamma$, and $$\Valley(\gamma\nearrow_5\gamma')=((2,1),(4,2)).$$
    Moreover, $\gamma''$ can be obtained by extending either the first or the second step of $\gamma$.
\end{exe}

\begin{rmk}\label{rmk:extension_parity}
Note that the extension valleys $\Valley(\gamma\nearrow_j\gamma')=(V,V')$ (from left to right) always have different $y$-coordinates. Furthermore, $y_{V'}>y_{V}$ if and only if the step extended is an up step ($j$ is odd). 
\end{rmk}

\begin{figure}[t]
    \centering
    \begin{subfigure}{0.45\textwidth}
        \centering
        \includegraphics[width=0.8\textwidth]{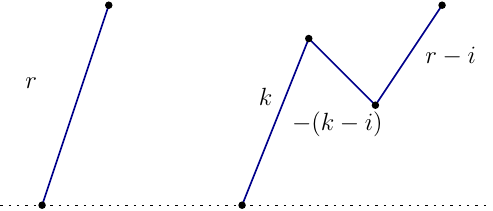}
        \subcaption*{The case $r>0$.}
    \end{subfigure}
    \begin{subfigure}{0.45\textwidth}
        \centering
        \includegraphics[width=0.8\textwidth]{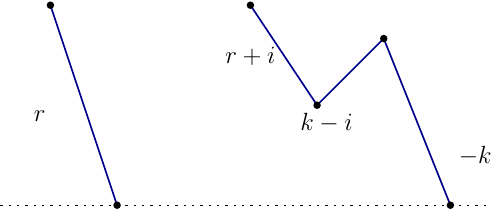}
        \subcaption*{The case $r<0$.}
    \end{subfigure}
    \caption{Extending a one step path.}
    \label{fig:add_valley}
\end{figure}

It turns out that commuting one particle operators with $\frac{-D_0}{M}$ corresponds to extending paths. 
\begin{lem}\label{lem:com_path_D0}
    Let $\gamma=(\gamma_1,\dots \gamma_{2n}) \in \bQ_{(n-1)} = \bR_{(1,0^{n-1})}$. Then
    \begin{equation*}\label{eq:cometa1}
		\ad_{\frac{D_0}{-M}}\left(\mcO_{(1,0^{n-1})}(\gamma)\right)=\sum_{1\leq i\leq 2n}\sum_{\gamma\nearrow_i\gamma'}\left([y_{V'}]_{(qt)^{-1}}-[y_{V}] _{(qt)^{-1}}\right)\mcO_{(1,0^{n})}(\gamma'),
    \end{equation*}
    where the second sum is taken over all paths $\gamma'$ obtained by extending the $i$-th step of $\gamma$, and where $\Valley(\gamma\nearrow_i\gamma')=(V,V').$
\end{lem}

\begin{proof}
Fix $r>0$. Then by \cref{lem:com_h}, we have
\begin{align*}
  \left[D_0,\mcO(r)\right]
  &=\sum_{k\geq 0}\Big[\mcO(k) \cdot \mcO(-k),\mcO(r)\Big]\\
  &=\sum_{k\geq 0}\mcO(k)\cdot \Big[\mcO(-k),\mcO(r)\Big]\\
  &=-M\cdot \sum_{k\geq 0}\sum_{1\leq i\leq k\wedge r}\left[i\right]_{qt}\mcO(k)\cdot \mcO(-(k-i))\mcO(r-i).
\end{align*}

Similarly if $r<0$, we have
\begin{align*}
  \left[D_0,\mcO(r)\right]
  &=\sum_{k\geq 0}\Big[\mcO(k)\cdot \mcO(-k),\mcO(r)\Big]\\
  &=\sum_{k\geq 0}\Big[\mcO(k),\mcO(r)\Big]\mcO(-k)\\
  &=-M\cdot\sum_{k\geq 0}\sum_{1\leq i\leq k\wedge -r}[-i]_{qt}\mcO(r+i)\cdot \mcO(k-i)\cdot \mcO(-k).
\end{align*}

Note that the summands in these equations correspond precisely to the possible ways of extending a step of degree $r$ in an alternating path as in \cref{def:path_extension}. More precisely, we have
\begin{align*}
	\ad_{\frac{D_0}{-M}}\left(\mcO_{(1,0^{n-1})}(\gamma)\right)
&=\frac{-1}{M}\sum_{1\leq i \leq 2n}\vweight_{(1,0^{n-1})}(\gamma)\mcO(\gamma_1)\dots \mcO(\gamma_{i-1})[D_0,\mcO(\gamma_i)]\mcO(\gamma_{i+1})\dots \mcO(\gamma_{2n})\\
&=\sum_{1\leq i \leq 2n}
\sum_{\gamma\nearrow_i\gamma'}[y_{V'}-y_V]_{qt}\vweight_{(1,0^{n-1})}(\gamma)\mcO(\gamma'_1)\dots  \mcO(\gamma'_{2n+2}).
\end{align*}
Notice that 
$$\vweight_{(1,0^{n})}(\gamma')=(qt)^{\height_{(1,0^{n})}(W)}\vweight_{(1,0^{n-1})}(\gamma),$$
where $W$ 
is the valley with the larger height among $(V,V'):=\Valley(\gamma\nearrow_i\gamma')$. As a consequence,
$$\vweight_{(1,0^{n})}(\gamma')=(qt)^{\max(\height_{(1,0^{n})}(V),\height_{(1,0^{n})}(V'))}\vweight_{(1,0^{n-1})}(\gamma)=(qt)^{\max(y_V,y_{V'})-1}\vweight_{(1,0^{n-1})}(\gamma).$$
We then have
\begin{align*}
	\ad_{\frac{D_0}{-M}}\left(\mcO_{(1,0^{n-1})}(\gamma)\right)
&=\sum_{1\leq i \leq 2n}
\sum_{\gamma\nearrow_i\gamma'}[y_{V'}-y_V]_{qt}(qt)^{1-\max(y_V,y_{V'})}\vweight_{(1,0^{n})}(\gamma')\mcO(\gamma'_1)\dots  \mcO(\gamma'_{2n+2})\\
&=\sum_{1\leq i \leq 2n}
\sum_{\gamma\nearrow_i\gamma'}[y_{V'}-y_V]_{qt}(qt)^{1-\max(y_V,y_{V'})}\mcO_{(1,0^{n})}(\gamma').
\end{align*}
We conclude using the fact that 
\begin{equation*}
  [y_{V'}-y_V]_{qt}(qt)^{1-\max(y_V,y_{V'})}=[y_{V'}]_{(qt)^{-1}}-[y_V]_{(qt)^{-1}}.\qedhere
\end{equation*}
\end{proof}

\begin{proof}[Proof of \cref{thm:A_n-P_n}]
We have
\begin{multline*}
	\frac{-1}{M}\left[D_{0},-e_1[X]\right]
        =\frac{-1}{M}\Big[\sum_{k\geq 0}\mcO(k)\cdot \mcO(-k),-e_1[X]\Big]\\
               =\sum_{k\geq 1}\mcO(k)\cdot \mcO(-(k-1))
               =\sum_{\gamma\in\bQ_{(0)}} \mcO_{(1)}(\gamma)
               =\mcQ_0,
             \end{multline*}
where the second equation is a consequence of~\eqref{eq:com_h}. 
Let $n\geq 1$. From \cref{lem:com_path_D0}, we have
 \[ \frac{-1}{M}\left[D_{0},\mcQ_{n-1}\right]
        =\sum_{\gamma \in \bQ_{(n-1)}} \ad^i_{\frac{D_0}{-M}}(\mcO_{(1,0^{n-1})}(\gamma))\\
        =\sum_{\gamma \in \bQ_{(n-1)}} \sum_{1\leq i\leq 2n} \sum_{\gamma\nearrow_i\gamma'} \left([y_{V'}]_{qt}-[y_{V}]_{qt}\right)\mcO_{(1,0^{n})}(\gamma'),\]
where $V$ and $V'$ are the two points of $\Valley(\gamma\nearrow_i\gamma')$ from left to right.
    \medskip
    
We now invert the two sums in the previous equation. To this purpose,
we note that for a fixed $\gamma'\in\bQ_{(n)}$ there is a one-to-one correspondence between:
\begin{itemize}
  \item pairs
$(V_j,V_{j+1})\subset
  \Valley(\gamma')$ of consecutive valleys of $\gamma'$ with different $y$-coordinates ($0\leq j\leq n$),
\item pairs $(\gamma, i)$, where $\gamma \in
  \bQ_{(n-1)}$, $1\leq i\leq 2n$ satisfying $\gamma \nearrow_i \gamma'$,
  \end{itemize}
  such that $\Valley(\gamma\nearrow_i\gamma')=(V,V').$ This correspondence is given by 
  \begin{equation*}
  \left\{
      \begin{array}{ll}
        i=1   & \text{if $j=0$}, \\
         i=2n  & \text{if $j=n$}\\
         i\in \left\{2j,2j+1\right\} &\text{if $1\leq j\leq n-1$},
      \end{array}\right.
  \end{equation*}
and the parity of $i$ is fixed by \cref{rmk:extension_parity}.
We then get
$$\frac{-1}{M}\left[D_{0},\mcQ_{n-1}\right]=\sum_{\gamma' \in \bQ_{(n)}}\mcO_{(1,0^n)}(\gamma')\sum_{0\leq j\leq n}\left([y_{V_{j+1}}]_{qt}-[y_{V_j}]_{qt}\right),$$
where $V_j$ is the valley $\gamma'$ with $x$-coordinate $2j$. Since $y_{V_0}=0$ and $y_{V_{n+1}}=1$, we deduce that
\begin{equation}\label{eq:com_D0_Pn}
	\frac{-1}{M}\left[D_{0},\mcQ_{n-1}\right] = \mcQ_{n} 
\end{equation}
as required.
\end{proof}

\subsection{Second commutation relation}\label{ssec:second_comm_rel}

We now prove a second commutation relation allowing us to obtain path operators with $\ell$ particles from path operators with $\ell-1$ particles.
Let $\alpha \in \mathbb{Z}_{\geq 0}^\ell$, and define the symmetrized operator:
$$\Qsym_\alpha:=\sum_{\sigma\in\mathfrak{S}_{\ell}}\mcQ_{\sigma(\alpha)},  
$$
where we denote the sequence $\sigma(\alpha):=(\alpha_{\sigma(1)},\dots,\alpha_{\sigma(\ell))})$. Let us introduce two simple operations on vectors $\alpha \in
\mathbb{Z}_{\geq 0}^\ell$: for $1\leq i\leq \ell$, we denote 
\begin{itemize}
  \item $\alpha^{\downarrow i}:=(\alpha_1,\dots,\alpha_{i-1},\alpha_{i+1},\dots,\alpha_\ell),$
\item $\alpha^{+ i} :=
  (\alpha_1,\dots,\alpha_{i-1},\alpha_{i}+ 1,\alpha_{i+1},\dots,\alpha_\ell).$
  \end{itemize}

The main result of this section is the following recursive formula for the operators $\Qsym_\alpha$.
\begin{thm}[Second commutation relation]\label{thm:P_commutation}
Let $\alpha \in \mathbb{Z}_{\geq 0}^\ell$ be a sequence of non-negative integers of length $\ell\geq2$. We then have
$$\Qsym_\alpha=\frac{1}{M}\sum_{1\leq i\leq \ell}\left[\ad^{\alpha_i}_{\frac{D_0}{-M}}\left(-e_1[X]\right),\Qsym_{\alpha^{\downarrow i}}\right].$$
\end{thm}
\begin{rmk}
we know from the first commutation relation (\cref{thm:A_n-P_n}), that if $\alpha_i>0$ then
$$\ad^{\alpha_i}_{\frac{D_0}{-M}}\left(-e_1[X]\right)=\mcQ_{\alpha_i-1}.$$
Intuitively, one can then think of the second commutation relation as follows: the symmetrized operator $\Qsym_\alpha$ is obtained by removing (in all possible ways) a part $\alpha_i$ from $\alpha$, and then commuting the operator obtained with $\mcQ_{\alpha_i-1}/M$.
\end{rmk}

To simplify notation,  we will write in this section
$$[n]:= [n]_{(qt)^{-1}}.$$

\begin{lem}\label{lem:com_path_e1}
Fix an alternating path $\gamma \in \bR_{\beta}$ of length $2n+2$. Then
\begin{equation*}\label{eq:come_1}
        \left[\frac{e_1[X]}{-M}, \mcO_{\beta}(\gamma)\right]=\sum_{1\leq i\leq n+1}\mathbbm{1}_{\gamma_{2i}<0}\mcO_{\beta^{+i}}(\gamma^{+i}),
    \end{equation*}
    where $\gamma^{+i} \in \bR_{\beta^{+i}}$ is the alternating path obtained from
    $\gamma$ by replacing the $2i$-th step by $\gamma^{+i}_{2i}=\gamma_{2i}+1$ if possible (i.e if $\gamma_{2i}\neq 0$).
\end{lem}
\begin{exe}
Consider the path $\gamma=(3,-1,2,-1,1,0,1,-1)\in \bR_{(1,2,0,1)} = \bQ_{(1,0,2,0)}$ given in~\cref{fig:RtoQ}.
    Then the path $\gamma^{+2}=(3,-1,2,0,1,0,1,-1) \in \bR_{(1,3,0,1)} = \bQ_{(1,0,0,2,0)}$ is well defined, while $\gamma^{+3}$ is not.
\end{exe}
\begin{proof}[Proof of \cref{lem:com_path_e1}]
    This is a consequence of the fact that
$$\vweight_{\beta}(\gamma)=\vweight_{\beta^{+i}}(\gamma^{+i}),$$
	and of the following commutation relation, obtained by taking $r=1$ in \cref{eq:com_h2}:
    \begin{equation*}
      \left[\frac{e_1[X]}{-M},\mcO(m)\right]=\left\{\begin{array}{ll}
         \mcO(m+1)&\text{if $m<0$},  \\
         0 & \text{if $m\geq 0$.}
    \end{array}\right.\qedhere  
    \end{equation*}
\end{proof}

 Recall from \cref{sub:path_op} that for any $\alpha \in \ZZ^n_{\geq 0}$ one has $\bQ_\alpha = \bR_{\psi^{-1}(\alpha)}$, where $\psi$ is a bijection between sequences in $\ZZ_{>0}\times \ZZ_{\geq 0}^{\ell-1}$ with size $n$ and sequences in $\ZZ_{\geq 0}^{n}$ with size $\ell-1$. This bijection switches from the parametrization by the number of particles over peaks to the distances between particles. Since we always have at least one particle over the first peak, the first particle will be always thought of as a \textit{frozen} particle. All the other particles are \textit{non-frozen}.

The following lemma allows us to relate some series of path operators where we have moved the position of one marked non-frozen particle by one step.
\begin{lem}\label{lem:move_peak}
For any $\alpha \in \mathbb{Z}_{\geq 0}^\ell$ and $1\leq i\leq \ell-1$, we have
\begin{equation}\label{eq:move_peak}
\sum_{\gamma \in\bQ_{\alpha^{+i}}}\left([\height_{\beta}(V_i)+1]-1\right)\mcO_{\beta}(\bfgamma) =\sum_{\gamma\in\bQ_{\alpha^{+(i+1)}}}[\height_{\beta'}(V_i')+1]\mcO_{\beta'}(\bfgamma),
\end{equation}
where $V_i$ (resp. $V'_i$) is the valley preceding (resp. succeeding) the $i+1$-th particle, $\beta = \psi^{-1}(\alpha^{+i})$, and $\beta' = \psi^{-1}(\alpha^{+(i+1)})$.
\end{lem}

\begin{proof}
First, notice that $\bQ_{\alpha^{+(i+1)}} \subset \bQ_{\alpha^{+i}}$. This is obvious, when we think that $\bQ_{\alpha^{+i}}$ is obtained from $\bQ_{\alpha^{+(i+1)}}$ by moving the
$i+1$-th particle to the right (note that this is well defined since the length between the $i+1$-th and the
$i+2$-th particle in $\bQ_{\alpha^{+(i+1)}}$ given by $\left(\alpha^{+(i+1)}\right)_{i+1}$ is at least $1$; see~\cref{fig:move_marked_peak}). In particular $\beta ' = \psi^{-1}(\alpha^{+(i+1)}) = (\beta_1,\dots,\beta_{k}+1,\beta_{k+1},\dots,\beta_{n+1})$ for some $1 \leq k \leq n$, and $\beta = \psi^{-1}(\alpha^{+i}) = (\beta_1,\dots,\beta_{k},\beta_{k+1}+1,\dots,\beta_{n+1})$, so that
$$\vweight_{\beta}(\gamma)=qt\cdot \vweight_{\beta'}(\gamma)$$
for any $\gamma\in \bQ_{\alpha^{+(i+1)}}$.
We deduce that 
\begin{equation}\label{eq:move_peak_2}
\mcO_{\beta}(\gamma)=qt\cdot \mcO_{\beta'}(\gamma).
\end{equation}

Moreover,
$\bQ_{\alpha^{+i}}\backslash\bQ_{\alpha^{+(i+1)}}$
consists of paths in $\bQ_{\alpha^{+i}}$ such that the valley
preceding the $i+1$-th particle has minimal $\beta$-height, i.e.~$\height_{\beta}(V_i)=0$.
In particular, $([\height_{\beta}(V_i)+1]-1) = 0$, so that the contribution of all these paths
in the sum on the left-hand side
of~\cref{eq:move_peak} is zero. Therefore this sum can be rewritten as
\begin{align*}
\sum_{\gamma \in \bQ_{\alpha^{+(i+1)}}}\left([\height_\beta(V_i)+1]-1\right)\mcO_\beta(\gamma)
&=\sum_{\gamma \in \bQ_{\alpha^{+(i+1)}}}\left([\height_{\beta'}(V'_i)+2]-1\right)\mcO_\beta(\gamma)\\
&=\sum_{\gamma \in \bQ_{\alpha^{+(i+1)}}} qt\left([\height(V'_i)_{\beta'}+2]-1\right)\mcO_{\beta'}(\gamma).
\end{align*}
We used here the fact that $\height_\beta(V_i)=\height_{\beta'}(V'_i)+1$ and \cref{eq:move_peak_2}. We conclude with the fact that $qt\left([n+1]-1\right)=[n]$, for any $n\geq 0$.\qedhere
\end{proof}

\begin{figure}
    \centering
    \begin{subfigure}[t]{0.45\textwidth}
\includegraphics[width=0.8\textwidth]{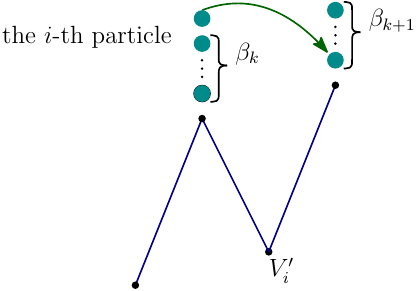}
\subcaption{An alternating path $\gamma \in \bQ_{\alpha^{+(i+1)}}$.}
\end{subfigure}
\begin{subfigure}[t]{0.45\textwidth}
\includegraphics[width=0.6\textwidth]{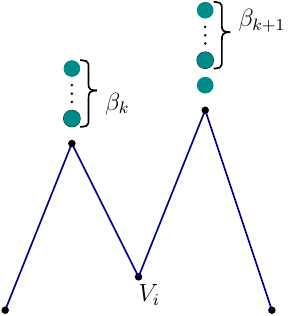}
\subcaption{The same path $\gamma \in \bQ_{\alpha^{+i}}$.}    
\end{subfigure}
    \caption{Moving the $i+1$-th particle of a path $\gamma \in \bQ_{\alpha^{+(i+1)}}$ one step to the right to consider the same path $\gamma \in \bQ_{\alpha^{+i}}$. The two peaks associated of $\gamma$ are respectively decorated by $\beta_k+1$ and $\beta_{k+1}$ particles, for $\psi^{-1}(\alpha^{+(i+1)}) = (\beta_1,\dots,\beta_{k}+1,\beta_{k+1},\dots,\beta_{n+1})$ and some $1 \leq k \leq n$.}
    \label{fig:move_marked_peak}
\end{figure}

\begin{prop}\label{prop:D0_P}
    Let $\alpha \in \left(\mathbb{Z}_{\geq 0}\right)^{\ell}$.
    Then $$\left[-D_{0}/M,\mcQ_{\alpha}\right]=\sum_{1\leq i \leq \ell}\mcQ_{\alpha^{+i}}.$$
\end{prop}
Notice that \cref{eq:com_D0_Pn} corresponds to the special case of paths with one decorating particle (i.e $\ell=1$) in the previous proposition.

Before proving this proposition we define path extensions for path decorated with arbitrarily many particles. This definition generalizes \cref{def:path_extension} in which we treated the one-particle case. We recall that in this case a path is extended  by ``inserting'' two steps following some rules, and the first particle is always on the first peak. 

\begin{defi}
    Let $\gamma \in \bR_{\beta}$ be an alternating path of length $2n$. We say that a pair $(\gamma',\beta')$ is obtained by extending the $i$-th step of the pair $(\gamma,\beta)$ and we write $(\gamma,\beta)\nearrow_i(\gamma',\beta')$, if we have an extension for the underlying paths $\gamma\nearrow_i\gamma'$ as in \cref{def:path_extension}, and if $\beta'$ is given by
    \begin{equation*}
        \beta'=\left\{\begin{array}{ll}
          (1,\beta_1-1,\beta_2,\dots,\beta_n)      &  \text{if $i=1$}\\
              (\beta_1,\dots,\beta_{\lfloor \frac{i}{2}\rfloor},0,\beta_{\lfloor \frac{i}{2}\rfloor+1},\dots,\beta_{n})& \text{if $i > 1$.}
        \end{array}\right.
    \end{equation*}
    In terms of particles, this means that the (non-frozen) particles are moved with respect to the following rules (see \cref{fig:add_marked_peak}): 
\begin{itemize}
    \item when we extend an up-step connected to a peak decorated by $s$ non-frozen particles, these particles all go to the right peak after the extension,
    \item when we extend a down-step connected to a peak decorated by $s$ particles, all of them go to the left peak after the extension.
\end{itemize}
Note that these rules for particle positions are in such a way that the created valley has always a non-negative $\beta'$-height, and in particular $\gamma'\in \bR_\beta$.
\end{defi}

\begin{proof}[Proof of \cref{prop:D0_P}]
Let $2n:=2|\alpha|+2$ be the length of a path in $\bQ_\alpha = \bR_\beta$. Recall also that $\ell=\ell(\alpha)$ is the number of particles of such path.
Using the same arguments as in
\cref{lem:com_path_D0}, we get that for any $\gamma\in \bQ_\alpha$
$$\left[-D_0/M, \mcO_\beta(\gamma)\right]=\sum_{1\leq i\leq
2n}\sum_{(\gamma,\beta)\nearrow_i(\gamma',\beta')}\mcO_{\beta'}(\gamma')\left([\theight_{\beta'}(V')]-[\theight_{\beta'}(V)]\right),$$
where $(V,V')=\Valley(\gamma\nearrow_i\gamma')$ and 
\begin{align*}
  \theight_{\beta')}(V)
  &=y_V-\#\text{non-frozen particles on the left of $V$}\\
  &=\left\{
\begin{array}{cc}
     \height_{\beta'}(V)+1&  \text{if $x_V>0$} \\
     0 &\text{if $x_V=0$}. 
\end{array}
\right.  
\end{align*}
We now sum over $\gamma\in\bQ_\alpha$ and we invert the sums as in the proof of \cref{thm:A_n-P_n}. We get
$$\left[-D_0/M, \mcQ_{\alpha}\right]=\sum_{\substack{1\leq i\leq
    \ell:\\ \alpha_i > 0 \text{ or } i\in\{1,\ell\} }}\left(\sum_{\gamma\in\bQ_{\alpha^{+i}}}\mcO_{\psi^{-1}(\alpha^{+i})}(\gamma)\sum_{(V,V')}\left([\theight_{\psi^{-1}(\alpha^{+i})}(V')]-[\theight_{\psi^{-1}(\alpha^{+i})}(V)]\right)\right),$$
where the third sum runs over all pairs of consecutive valleys
$(V,V')$ (in this order) lying:
\begin{itemize}
    \item before the second particle if $i=1$,
    \item after the last particle if $i=\ell$,
    \item between the $i$-th and the $i+1$-th particle if $1<i<\ell$.
\end{itemize}
 Note that
this set is not empty, because $(\alpha^{+i})_i >1$ whenever $\alpha_i>0$. The telescopic sum gives
\[ \sum_{(V,V')}\left([\theight_{\psi^{-1}(\alpha^{+i})}(V')]-[\theight_{\psi^{-1}(\alpha^{+i})}(V)]\right) =
  [\theight_{\psi^{-1}(\alpha^{+i})}(V_i)]-[\theight_{\psi^{-1}(\alpha^{+i})}(V'_{i-1})]\]
where $V_i$ (resp. $V'_i$) is the valley preceding (resp. succeeding) the $i+1$-th particle, $\beta = \psi^{-1}(\alpha^{+i})$, and $\beta' = \psi^{-1}(\alpha^{+(i+1)})$, with the convention that
$V_0' = (0,0)$, and $V_\ell = (2n,\ell)$ are respectively the first and the last valleys. Hence,
\begin{align*}
\left[-D_{0}/M,\mcQ_{\alpha}\right]&=\sum_{\substack{1\leq i\leq
    \ell,\\\alpha_i > 0 \text{ or } i\in\{1,\ell\}}} \left(\sum_{\gamma\in\bQ_{\alpha^{+i}}}\left([\theight_{\psi^{-1}(\alpha^{+i})}(V_i)]-[\theight_{\psi^{-1}(\alpha^{+i})}(V'_{i-1})]\right)\mcO_{\psi^{-1}(\alpha^{+i})}(\gamma)\right).
\end{align*}
Notice that the condition $\alpha_i > 0$ can be dropped. Indeed,
when $1<i<\ell$, $\alpha_i = 0$ and $\gamma\in\bQ_{\alpha^{+i}} $ we have $V_i =
V'_{i-1}$. Therefore
\begin{align*}
  \left[-D_{0}/M,\mcQ_{\alpha}\right]=&\sum_{\gamma\in\bQ_{\alpha^{+\ell}}}[1]\mcO_{\psi^{-1}(\alpha^{+\ell})}(\gamma)-\sum_{\gamma\in\bQ_{\alpha^{+1}}}[0]\mcO_{\psi^{-1}(\alpha^{+1})}(\gamma)\\
  &+\sum_{1\leq
     i\leq\ell-1}\bigg(\sum_{\gamma\in\bQ_{\alpha^{+i}}}[\theight_{\psi^{-1}(\alpha^{+i})}(V_i)]\mcO_{\psi^{-1}(\alpha^{+i})}(\gamma)\\
	 &-\sum_{\gamma\in\bQ_{\alpha^{+(i+1)}}}[\theight_{\psi^{-1}(\alpha^{+(i+1)})}(V'_i)]\mcO_{\psi^{-1}(\alpha^{+(i+1)})}(\gamma)\bigg)\\
     =& \sum_{1\leq
i\leq\ell}\sum_{\gamma\in\bQ_{\alpha^{+i}}}\mcO_{\psi^{-1}(\alpha^{+i})}(\gamma),
\end{align*}
where we used \cref{lem:move_peak} to obtain the last equality. This finishes the proof of the proposition.
\end{proof}
\begin{figure}[t]
    \centering
    \begin{subfigure}{0.45\textwidth}
        \centering
        \includegraphics[width=0.8\textwidth]{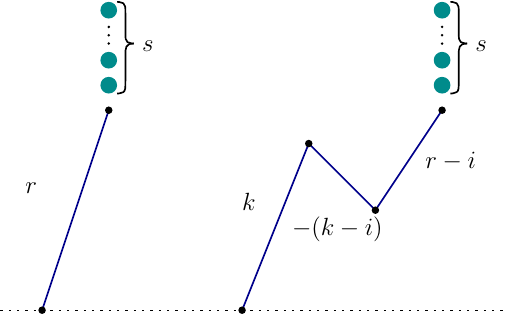}
        \subcaption*{The case $r>0$.}
    \end{subfigure}
    \begin{subfigure}{0.45\textwidth}
        \centering
        \includegraphics[width=0.8\textwidth]{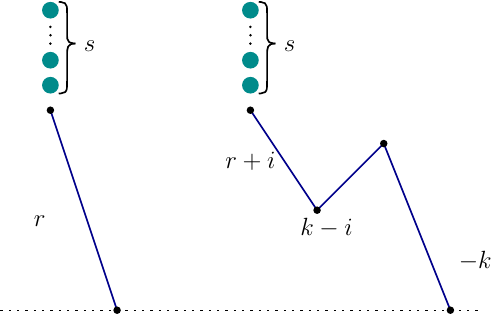}
        \subcaption*{The case $r<0$.}
    \end{subfigure}
    \caption{Moving particles when commuting $-D_{0}/M$ with a step of a
      decorated path. Here, the step has size $r$ and is incident to a peak decorated with $s$ non-frozen particles, for some $s\geq0$.}
    \label{fig:add_marked_peak}
  \end{figure}

\begin{prop}\label{prop:commutation_P}
   Fix two integers $m\geq 0, \ell \geq 1$, and let $\alpha\in\ZZ_{\geq 0}^\ell$ be a sequence of non-negative integers. Then 
    \begin{multline*}
      \frac{1}{-M}\left[\ad^{m}_{\frac{D_0}{-M}}\left(-e_1[X]\right),\mcQ_{\alpha}\right]=
      \sum_{1\leq i\leq\ell-1} \mcQ_{\alpha_1,\dots,\alpha_i,m,\alpha_{i+1},\dots,\alpha_\ell}
      \\+\sum_{1\leq i \leq\ell}\sign(m-\alpha_i-1)\sum_{\substack{r_1,r_2\geq m\wedge(\alpha_i+1)\\{r_1+r_2=m+\alpha_i}}}\mcQ_{\alpha_1,\dots,\alpha_{i-1},r_1,r_2,\alpha_{i+1},\dots,\alpha_\ell}.  
    \end{multline*}
  \end{prop}
\begin{proof}
We start by proving the case $m=0$ and then we use \cref{prop:D0_P} to
prove by induction the proposition for any $m$.

\medskip

Let us consider the following sum
\[ \sum_{1\leq i\leq \ell}\sum_{\substack{r_1,r_2\geq
    0\\{r_1+r_2=\alpha_i}}}\mcQ_{\alpha_1,\dots,\alpha_{i-1},r_1,r_2,\alpha_{i+1},\dots,\alpha_{\ell}}.\]
We can interpret it combinatorially as follows: for each $1 \leq i\leq
\ell$ we add a new particle between two old particles
separated by $\alpha_i$ peaks in all possible places, including a
particle over the first old particle, and a
particle below the second old particle, which gives $\alpha_i+1$ possible
places. This gives in total $\sum_{1 \leq i \leq \ell}(\alpha_i+1) =
|\alpha|+\ell$ possible positions. \cref{lem:com_path_e1} implies that
the commutation of $\mcO_{\psi^{-1}(\alpha)}(\gamma)$ with $-e_1[X]/M$ corresponds to
adding a new particle to all the peaks, which gives $|\alpha|+1$
positions. The difference between these two interpretations comes from
the fact that in the first sum we always add a particle below the
second old particle separated by $\alpha_i$ and over the first particle
separated by $\alpha_{i+1}$, which is the same particle. Therefore the
correct quantity corresponding to the
commutator $\left[\frac{e_1[X]}{M},\mcQ_{\alpha}\right] $ is equal
to the following correction of the first sum:
  \begin{align*}
\left[\frac{e_1[X]}{M},\mcQ_{\alpha}\right]
   =\sum_{1\leq i\leq\ell-1} \mcQ_{\alpha_1,\dots,\alpha_i,0,\alpha_{i+1},\dots,\alpha_\ell}-\sum_{1\leq i\leq \ell}\sum_{\substack{r_1,r_2\geq 0\\{r_1+r_2=\alpha_i}}}\mcQ_{\alpha_1,\dots,\alpha_{i-1},r_1,r_2,\alpha_{i+1},\dots,\alpha_{\ell}}.   
  \end{align*}
This finishes the proof of the proposition for $m=0$.

We now fix $m\geq 0$ and we assume that the proposition holds for $m$. The Jacobi identity gives:
\begin{align*}
	\left[\frac{\ad^{m+1}_{\frac{D_0}{-M}}\left(-e_1[X]\right)}{-M},\mcQ_{\alpha}\right]
  =&\left[\left[\frac{D_0}{-M},\frac{\ad^{m}_{\frac{D_0}{-M}}\left(-e_1[X]\right)}{-M}\right],\mcQ_{\alpha}\right]\\
  =&\left[\frac{D_0}{-M},\left[\frac{\ad^{m}_{\frac{D_0}{-M}}\left(-e_1[X]\right)}{-M},\mcQ_{\alpha}\right]\right]-\left[\frac{\ad^{m}_{\frac{D_0}{-M}}\left(-e_1[X]\right)}{-M},\left[\frac{D_0}{-M},\mcQ_{\alpha}\right]\right].
\end{align*}
Combining the recurrence assumption and \cref{prop:D0_P}, we obtain
\begin{multline*}
	\frac{1}{-M}\left[\ad^{m}_{\frac{D_0}{-M}}\left(-e_1[X]\right),\mcQ_{\alpha}\right]=
  \sum_{1\leq i\leq\ell-1}\mcQ_{\alpha_1,\dots,\alpha_i,m+1,\alpha_{j+1},\dots,\alpha_\ell}+\sum_{1\leq i \leq\ell}\sign(m-\alpha_i-1)\\
  \sum_{\substack{r_1,r_2\geq m\wedge(\alpha_i+1)\\{r_1+r_2=m+\alpha_i}}}\left(\mcQ_{\alpha_1,\dots,\alpha_{i-1},r_1+1,r_2,\alpha_{i+1},\dots,\alpha_\ell}+\mcQ_{\alpha_1,\dots,\alpha_{i-1},r_1,r_2+1,\alpha_{i+1},\dots,\alpha_\ell}\right)\\
  -\sum_{1\leq i \leq\ell}\sign(m-\alpha_i-2)\sum_{\substack{r_1,r_2\geq m\wedge(\alpha_i+2)\\{r_1+r_2=m+\alpha_i+1}}}\mcQ_{\alpha_1,\dots,\alpha_{i-1},r_1,r_2,\alpha_{i+1},\dots,\alpha_\ell}.
\end{multline*}
By distinguishing the four cases $\alpha_i>m-1$, $\alpha_i=m-1$, $\alpha_i=m-2$ and $\alpha_i<m-2$, one check that we obtain the required formula:
\begin{align*}
	\frac{1}{-M}\left[\ad^{m}_{\frac{D_0}{-M}}\left(-e_1[X]\right),\mcQ_{\alpha}\right]
=&\sum_{1\leq i\leq\ell-1}\mcQ_{\alpha_1,\dots,\alpha_i,m+1,\alpha_{i+1},\dots,\alpha_\ell}\\
  &+\sum_{1\leq i \leq\ell}\sign(m-\alpha_i)\sum_{\substack{r_1,r_2\geq (m+1)\wedge(\alpha_i+1)\\{r_1+r_2=m+1+\alpha_i}}}\mcQ_{\alpha_1,\dots,\alpha_{i-1},r_1,r_2,\alpha_{i+1},\dots,\alpha_\ell}.
\end{align*}
This finishes the proof of the proposition.
\end{proof}

Let us now prove \cref{thm:P_commutation}, which gives the second commutation relation announced at the beginning of the section.
\begin{proof}[Proof of \cref{thm:P_commutation}]

Fix $\ell \geq2$ and $\alpha\in\ZZ_{\geq 0}^{\ell}$. We have
\begin{align*}
        \frac{1}{M}\sum_{1\leq i\leq \ell}\left[\ad^{\alpha_i}_{\frac{D_0}{-M}}\left(-e_1[X]\right),\Qsym_{\alpha^{\downarrow i}}\right]
        &=\frac{1}{M}\sum_{1\leq i\leq \ell}\sum_{\sigma\in\mathfrak{S}_{\ell-1}}\left[\ad^{\alpha_i}_{\frac{D_0}{-M}}\left(-e_1[X]\right),\mcQ_{\sigma(\alpha^{\downarrow i})}\right]\\
        &=\frac{1}{M}\sum_{\sigma\in\mathfrak{S}_\ell}\left[\ad^{\alpha_{\sigma(\ell)}}_{\frac{D_0}{-M}}\left(-e_1[X]\right),
        \mcQ_{\sigma(\alpha^{\downarrow \ell})}\right].
\end{align*}
Applying \cref{prop:commutation_P}, this is also equal to
\begin{multline*}
	-\sum_{\sigma\in\mathfrak{S}_{\ell}} \sum_{1\leq j\leq \ell-2}\mcQ_{\alpha_{\sigma(1)},\dots,\alpha_{\sigma(j)},\alpha_{\sigma(\ell)},\alpha_{\sigma(j+1)},\dots,\alpha_{\sigma(\ell-1)}}\\
-\sum_{\sigma\in\mathfrak{S}_{\ell}}
        \sum_{1\leq j\leq \ell-1}
        \sign\left(\alpha_{\sigma(\ell)}-1-\alpha_{\sigma(j)}\right)
        \sum_{\substack{r_1,r_2\geq \alpha_{\sigma(\ell)} \wedge\left(\alpha_{\sigma(j)}+1\right)\\{r_1+r_2=\alpha_{\sigma(\ell)}+\alpha_{\sigma(j)}}}}\mcQ_{\alpha_{\sigma(1)},\dots,\alpha_{\sigma(j-1)},r_1,r_2,\alpha_{\sigma(j+1)},\dots}. \label{eq:thm:comP}
    \end{multline*}
The first term on the RHS gives 
$$-\sum_{1\leq j\leq
  \ell-2}\sum_{\sigma\in\mathfrak{S}_{\ell}}\mcQ_{\sigma \cdot (\ell, \ell-1,\dots,j+1)(\alpha)}
= -(\ell-2)\sum_{\sigma\in\mathfrak{S}_{\ell}}
\mcQ_{\sigma(\alpha)}=-(\ell-2)\Qsym_{\alpha},$$
where $(\ell, \ell-1,\dots,j+1)$ is the permutation of $\mathfrak{S}_\ell$ with cyclic notation.
Let us now compute the second term. For a fixed $1\leq j\leq \ell-1$,
define a bijection $\mathfrak{S}_{\ell} \rightarrow
\mathfrak{S}_{\ell}$ given by $\sigma \mapsto \sigma' := \sigma(j,\ell)$, where
$(j,\ell)$ is the transposition exchanging $j$ and $\ell$. It is easy
to check that a contribution to the second term in the RHS that comes from
$\sigma$ and the associated $\sigma'$ is equal to
$$\mcQ_{\sigma\cdot(\ell,\ell-1,\dots,j+1)(\alpha)}+\mcQ_{\sigma\cdot(\ell,\ell-1,\dots,j+1,j)(\alpha)}.$$
As a consequence, the sum over
all permutations $\sigma$ contributing to the second term on the RHS
with a fixed $1\leq j\leq \ell-1$ is equal to $\Qsym_{\alpha}$, so
that the total contribution is equal to $(\ell-1) \Qsym_{\alpha}$. In conclusion, 
        \begin{equation*}
			\frac{1}{M}\sum_{1\leq i\leq \ell}\left[\ad^{\alpha_i}_{\frac{D_0}{-M}}\left(-e_1[X]\right),\Qsym_{\alpha^{\downarrow i}}\right]=\big(-(\ell-2)+(\ell-1)\big)\Qsym_{\alpha}=\Qsym_{\alpha}
        \end{equation*}
        as claimed.
\end{proof}

\subsection{Combinatorial formula in terms of path operators}\label{ssec:proof_comb_formula}

We are ready to prove the combinatorial formula for the operators $\A_F^{(\ell)}$ announced in \cref{thm:dif_eq''}.

\begin{proof}[Proof of \cref{thm:dif_eq''}]
	
		We prove the theorem by induction. For $\ell=1$, we know from \eqref{def:AFell} and \cref{thm:A_n-P_n}, that
	\begin{equation*}\label{eq:A-Q}
	  \A_F^{(1)}=a_0\mcQ_0+a_1\mcQ_{1}\dots  
	\end{equation*}
	giving the theorem for $\ell=1$.

		We now assume that \eqref{eq:diff_eq'} holds for some $\ell\geq 1$. We then have from \eqref{def:AFell} and the induction assumption
		\begin{align*}
	   \A_F^{(\ell+1)}
		  =\frac{1}{M}\left[\A_{F},\A_F^{(\ell)}\right]
		  &=\frac{1}{M}\sum_{j\geq 0}\sum_{\alpha \in \ZZ_{\geq 0}^\ell}
		  a_{j} a_\alpha\left[\ad^j_{\frac{D_0}{-M}}(e_1[X]),\mcQ_{\alpha}\right],  
		\end{align*}
		where we use \cref{thm:A_n-P_n} to obtain the last equality. This can be written as follows
		\begin{align*}
		  \A_F^{(\ell+1)}
		  &=\frac{1}{M}\sum_{0\leq \nu_1\leq\dots\leq \nu_{\ell+1}}
			\frac{1}{|\Aut(\nu)|} a_\nu \sum_{\sigma\in\mathfrak
S_{\ell+1}}\left[\ad^{\nu_{\sigma(\ell+1)}}_{\frac{D_0}{-M}}(e_1[X]),\mcQ_{\sigma (\nu^{\downarrow(\ell+1)})}\right],
		\end{align*}		
  where $\Aut(\nu)<\mathfrak{S}_{\ell+1}$ is the stabilizer of the partition $\nu$.

		Applying \cref{thm:P_commutation}, we get
	\[      \A_F^{(\ell+1)}=\sum_{0\leq \nu_1\leq\dots\leq \nu_{\ell+1}}\frac{1}{|\Aut(\nu)|}a_\nu\sum_{\sigma\in\mathfrak
			S_{\ell+1}}\mcQ_{\sigma (\nu)}=\sum_{\alpha \in
			\mathbb{Z}_{\geq 0}^{\ell+1}}a_{\alpha}\mcQ_{\alpha}. \]
		This finishes the proof of the theorem.
	\end{proof}

\section{Applications: relation to the extended delta conjecture}\label{sec:relation_delta}

Define the operators $\Delta_{h_n}$ and $\Delta_{e_n}$ by their generating series:
\begin{equation*}
\label{eq:Pi-Delta}
	\Pi_{1+uz}=\sum_{n\geq 0}u^{n}\Delta_{e_n}, \quad \Pi_{(1-vz)^{-1}}=\sum_{n\geq 0}v^{n}\Delta_{h_n}.
\end{equation*}
In particular, they act diagonally on the basis of transformed Macdonald polynomials.
Define also the operator
\[ \Delta'_{e_n} = \sum_{i=0}^n (-1)^i\Delta_{e_{n-i}}.\]
We prove that the symmetric function $\Delta_{h_l}\Delta_{e_{k-1}}\cdot e_{n}$, which is the main object of interest in the extended delta conjecture, has an explicit formula in terms of the action of the path operators on $1$. This result together with \cref{thm:explicit_formula2} is equivalent to the key result~\cite[Theorem 4.4.1]{BlasiakHaimanMorsePunSeelinger2023b} in proving the extended delta conjecture. The proof in~\cite{BlasiakHaimanMorsePunSeelinger2023b} is based on the intermediate results describing the structure of the elliptic Hall algebra and its representation on $\Lambda$. We give a new proof that relies directly on combinatorial properties of path operators.

\begin{thm}\label{thm:ExtDelta}
    Fix three integers $l \geq 0$ and $0 < k \leq n$, and define the function $F\in S_{k+l}$
    $$F:=\operatorname{Sym}\left(\frac{z_1 \dotsm z_{k+l}}{\prod_{i=1}^{l+k-1} (1-qtz_{i+1}/z_i)}h_{n-k}(z_1,\dots,z_{k+l})e_l(z_2^{-1},\dots,z_{l+k}^{-1}) \prod_{1\leq i<j\leq \ell} \omega(z_j/z_i)\right).$$
    We have
    \begin{align}\label{eq:extended_delta_conj}
      (-1)^{n}\Delta_{h_l}\Delta'_{e_{k-1}}\cdot e_{n}=\sum_{\substack{\beta \in \ZZ_{\geq 0}^{k+l-1},\\|\beta| = n-k}}\sum_{\substack{\beta' \in \{0,1\}^{k+l-1},\\|\beta'| = l}}\mcR_{\beta+1^{k+l}-(0,\beta')}\cdot 1
      =\widehat{F}\cdot 1,
    \end{align}
       where $\widehat{F}$ is the operator on $\Lambda$ given by the action of the shuffle algebra (see \cref{def:shuffle_algebra}). 
\end{thm}

\begin{proof}
Notice that
\[(-1)^{n}e_{n}=D_{n}\cdot 1=\mcQ_{0^{n}}\cdot 1.\]
Therefore, we can rewrite $\Delta_{h_l}\Delta'_{e_{k-1}}\cdot e_{n}$ as follows:
$$(-1)^{n}\Delta_{h_l}\Delta'_{e_{k-1}}\cdot e_{n}=\sum_{i=1}^{k}(-1)^{i-1}[v^{l}u^{k-i}]\Pi_{\frac{1+uz}{1-vz}}\cdot \mcQ_{0^{n}}\cdot\left(\Pi_{\frac{1+uz}{1-vz}}\right)^{-1}\cdot 1.$$
Moreover, we know from \cref{thm:dif_eq''} that $\mcQ_{0^{n}}=\A_{K}^{(n)}$, where $K(z)=1$ is the constant polynomial. \cref{lem:DeltaG_A+} implies that 
$\Pi_G\cdot \A_{K}^{(n)}\cdot \Pi_G^{-1}=\A_{KG}^{(n)}=\A_{G}^{(n)}.$  We then get 
\begin{equation}\label{eq:extended_delta}
  (-1)^{n}\Delta_{h_l}\Delta'_{e_{k-1}}\cdot e_{n}=\sum_{i=1}^{k}(-1)^{i-1}[v^{l}u^{k-i}]\A_{\frac{1+uz}{1-vz}}^{(n)}\cdot 1.  
\end{equation}
By reapplying \cref{thm:dif_eq''}, we have that
\[ [v^{l}u^{k-i}]\A_{\frac{1+uz}{1-vz}}^{(n)}\cdot 1=\sum_{\substack{\alpha \in \ZZ_{\geq 0}^n,\\|\alpha| = l}}\sum_{\substack{\alpha' \in \{0,1\}^n,\\|\alpha'| = k-i}}\mcQ_{\alpha+\alpha'}\cdot 1 = \sum_{\substack{\alpha_1,\dots,\alpha_{n} \geq 0,\\|\alpha| = l+k-i}}\binom{\ell(\alpha^+)}{k-i}\mcQ_{\alpha_1,\dots,\alpha_{n}}\cdot 1,\]
where we use the notation $\ell(\alpha^+)$ to denote the number of non-zero parts of $\alpha$. Recall the map $\psi$ introduced in \cref{sub:path_op} that relates the operators $\mcQ_\alpha = \mcR_\beta$, where $\psi^{-1}(\alpha)=:\beta \in \ZZ_{>0}\times \ZZ^{l+k-i}_{\geq 0}$ for $\alpha$ such that $|\alpha|=l+k-i$. Notice that 
\begin{itemize}
	\item if $\alpha_{n} = 0$, then $\ell(\alpha^+) = \ell(\beta^+)-1$, and $\beta_{l+k-i+1} > 0$,
	\item if $\alpha_{n} > 0$, then $\ell(\alpha^+) = \ell(\beta^+)$, and $\beta_{l+k-i+1} = 0$.
\end{itemize}
Therefore $[v^{l}u^{k-i}]\A_{\frac{1+uz}{1-vz}}^{(n)}\cdot 1$ can be rewritten as follows:
\begin{multline*}
	\sum_{\substack{\beta_1,\beta_{l+k-i+1}>0,\\ \beta_2,\dots,\beta_{l+k-i} \geq 0,\\|\beta| = n}}\binom{\ell(\beta^+)-1}{k-i}\mcR_{\beta_1,\dots,\beta_{l+k-i+1}}\cdot 1  + \sum_{\substack{\beta_1>0,\\ \beta_2,\dots,\beta_{l+k-i} \geq 0,\\|\beta| = n}}\binom{\ell(\beta^+)}{k-i}\mcR_{\beta_1,\dots,\beta_{l+k-i},0}\cdot 1\\
	= \sum_{\substack{\beta_1>0,\\ \beta_2,\dots,\beta_{l+k-i+1} \geq 0,\\|\beta| = n}}\binom{\ell(\beta^+)-1}{k-i}\mcR_{\beta_1,\dots,\beta_{l+k-i+1}}\cdot 1+\sum_{\substack{\beta_1>0,\\ \beta_2,\dots,\beta_{l+k-i} \geq 0,\\|\beta| = n}}\binom{\ell(\beta^+)-1}{k-i-1}\mcR_{\beta_1,\dots,\beta_{l+k-i}}\cdot 1.
\end{multline*}
The second equality comes from the binomial identity
\[ \binom{\ell(\beta^+)}{k-i} = \binom{\ell(\beta^+)-1}{k-i}+\binom{\ell(\beta^+)-1}{k-i-1}\]
and from the fact that $\mcR_{\beta,0}\cdot 1 = \mcR_{\beta}\cdot 1$, which is a consequence of the identity $\mathcal{O}(-k) \cdot 1 = \delta_{k,0}$ for $k \geq 0$. Applying this in \cref{eq:extended_delta}, we end up with the following formula:
\begin{align*}
	(-1)^{n}\Delta_{h_l}\Delta'_{e_{k-1}}\cdot e_{n} 
    &= \sum_{\substack{\beta_1>0,\\ \beta_2,\dots,\beta_{l+k} \geq 0,\\|\beta| = n}}\binom{\ell(\beta^+)-1}{k-1}\mcR_{\beta_1,\dots,\beta_{l+k}}\cdot 1\\
    &= \sum_{\substack{\beta \in \ZZ_{\geq 0}^{k+l},\\|\beta| = n-k}}\sum_{\substack{\beta' \in \{0,1\}^{k+l-1},\\|\beta'| = l}}\mcR_{\beta+1^{k+l}-(0,\beta')}\cdot 1.
\end{align*}
To obtain the second part of \cref{eq:extended_delta_conj}, we use \cref{cor:Negut} to write:
$$\sum_{\substack{\beta \in \ZZ_{\geq 0}^{k+l},\\|\beta| = n-k}}\sum_{\substack{\beta' \in \{0,1\}^{k+l-1},\\|\beta'| = l}}\mcR_{\beta+1^{k+l}-(0,\beta')}=\widehat{F},$$
with 
$$F=\sum_{\substack{\beta \in \ZZ_{\geq 0}^{k+l},\\|\beta| = n-k}}\sum_{\substack{\beta' \in \{0,1\}^{k+l-1},\\|\beta'| = l}}\operatorname{Sym}\left(\frac{z^{\beta+1^{k+l}-(0,\beta')}}{\prod_{i=1}^{l+k-1} (1-qtz_{i+1}/z_i)} \prod_{1\leq i<j\leq \ell} \omega(z_j/z_i)\right),$$
where $z^\alpha:=z_1^{\alpha_1}\dots z_{l+k}^{\alpha_{l+k}}.$
We conclude using the definitions of homogeneous and elementary symmetric functions.
\end{proof}

Notice that there is an extra factor $(-1)^n$ in \cref{thm:ExtDelta} compared to \cite{BlasiakHaimanMorsePunSeelinger2023} due to the fact that our conventions for the operators $D_a$ differ by a factor $(-1)^a$.

\begin{rmk}
In~\cite[Definition 4.3.1]{BlasiakHaimanMorsePunSeelinger2023b}, the authors use lattice paths (different from the ones we consider here) to define an involutive correspondence between some sequences of nonnegative integers. It is not very hard to see that this correspondence is precisely the map $\psi$ introduced in \cref{sub:path_op}. Using the notation from \cite{BlasiakHaimanMorsePunSeelinger2023b}, they introduce the transposed Negut elements $E_\alpha$, and they prove in~\cite[Proposition 4.3.3]{BlasiakHaimanMorsePunSeelinger2023b} that $E_\alpha = D_\beta$, where $\psi(\beta) = \alpha$. Thus, \cref{cor:Negut} implies that $\mcQ_{\beta} = (-1)^{|\beta|}E_{\beta}$, and consequently we also have an explicit combinatorial formula for the representation of the transposed Negut elements.
\end{rmk}

% \subsection*{Acknowledgements}
\bibliographystyle{amsalpha}
\bibliography{biblio2015}
\end{document}